\documentclass[12pt, reqno]{amsart}
\usepackage{amsfonts}
\usepackage{amsthm}
\usepackage{amscd}
\usepackage{amssymb}
\usepackage{graphics}
\usepackage{amsmath}
\usepackage[english]{babel}
\usepackage{hyperref}
\usepackage{bbm}
\usepackage{yfonts}
\usepackage{mathrsfs}
\usepackage{color}
\usepackage{epsfig}
\usepackage{graphicx}
\allowdisplaybreaks[4]
\DeclareFontFamily{OML}{rsfs}{\skewchar\font'177}
\DeclareFontShape{OML}{rsfs}{m}{n}{ <5> <6> rsfs5 <7> <8> <9> rsfs7
<10> <10.95> <12> <14.4> <17.28> <20.74> <24.88> rsfs10 }{}
\DeclareMathAlphabet{\mathfs}{OML}{rsfs}{m}{n}

\newtheorem{prop}{Proposition}[section]

\newtheorem{thm}[prop]{Theorem}
\newtheorem{lem}[prop]{Lemma}

\newtheorem{conj}[prop]{Conjecture}

\newtheorem{defn}[prop]{Definition}

\newtheorem{prob}[prop]{Problem}

\newtheorem{clm}[prop]{Claim}

\newtheorem{ques}[prop]{Question}

\newtheorem{ass}[prop]{Assumption}

\newcommand{\BE}{{\mathbb{E}}}

\newcommand{\BN}{{\mathbb{N}}}

\newcommand{\BP}{{\mathbb{P}}}

\newcommand{\BR}{{\mathbb{R}}}

\newcommand{\BZ}{{\mathbb{Z}}}

\newcommand{\FB}{{\mathfrak{B}}}

\newcommand{\CE}{{\mathcal{E}}}

\newcommand{\CP}{{\mathcal{P}}}
\newcommand{\CQ}{{\mathcal{Q}}}

\newcommand{\ind}{{\mathbbm{1}}}

\newcommand{\bae}{\begin{equation}\begin{aligned}}
\newcommand{\eae}{\end{aligned}\end{equation}}

\newcommand{\om}{{\omega}}
\newcommand{\Om}{{\Omega}}
\newcommand{\si}{{\sigma}}

\newcommand{\gth}{{\theta}}

\newcommand{\ep}{{\epsilon}}

\begin{document}
\numberwithin{equation}{section} \numberwithin{figure}{section}

\newtheorem*{theo}{Theorem}
\newtheorem*{propo}{Proposition}

\title{Random Walk on Discrete Point Processes}
\author{Ron Rosenthal}

\maketitle

\begin{abstract}
We consider a model for random walks on random environments (RWRE)
with random subset of $\BZ^d$ as the vertices, and uniform
transition probabilities on $2d$ points (two "coordinate nearest
points" in each of the $d$ coordinate directions). We prove that the
velocity of such random walks is almost surely $0$, and give partial
characterization of transience and recurrence in the different
dimensions. Finally we prove Central Limit Theorem (CLT) for such
random walks, under a condition on the distance between coordinate
nearest points.
\end{abstract}



\section{Introduction}

\subsection{Background}
$~$\\

Random walk on random environments is the object of intensive
mathematical research for more then 3 decades. It deals with models
from condensed matter physics, physical chemistry, and many other
fields of research. The common subject of all models is the
investigation of movement of particles in an inhomogeneous media. It
turnes out that the randomness of the media (i.e. the environment)
is responsible for some unexpected results, especially in large
scale behavior. In the general case, the random walk takes place in
a countable graph $(V,E)$, but the most investigated models deals
with the graph of the $d$-dimensional integer lattice, (i.e.
$V=\BZ^d$). For some of the results on those models see \cite{Ze03},
\cite{Sz00}, \cite{Hu96} and \cite{Re05}. The definition of RWRE
involves two steps: First the environment is randomly chosen by some
given probability, then the random walk, which takes place on this
given fixed environment, is a Markov chain with transition
probabilities that depend on the environment. We note that the
environment is kept fixed and does not evolve during the random
walk, and that the random walk, given an environment, is not
necessarily reversible. The questions on RWRE come in two major
types: quenched, in which the walk is distributed according to a
given typical environment, and annealed, in which the distribution
of the walk is taken according to an average on the environments.
The two main differences between the quenched and the annealed are:
First the quenched is Markovian, while the annealed distribution is
usually not. Second, in most of the models we assume some kind of
translation invariance on the environments and therefore annealed is
usually translation invariance while quenched is not. In contrast to
most of the models for RWRE on $\BZ^d$, this work deals with non
nearest neighbor random walks. The subject of non nearest neighbor
random walks has not been systematically studied. For results on
long range percolation see \cite{Be01}. For literature on the
subject in the one dimensional case see \cite{BG08},\cite{Br02},
\cite{CS09}. For some results on bounded non nearest neighbors see
\cite{Ke84}. For some results that are valid in that general case
see \cite{Va02} and \cite{caputo2009invariance}. For recurrence and transience criteria for random
walks on random point processes, with transition probabilities
between every two points proportional to their distance, see
\cite{CFG8}. Our model also has the property that the random walk is
reversible. For some results in this topic see \cite{BBHK},
\cite{BP07}, \cite{MP07} and \cite{SS09}.


%
%

\subsection{The Model}
$~$\\

Let $\BZ^d$ be the $d$-dimensional lattice of integers. We define
$\Om=\{0,1\}^{\BZ^d}$ and $\mathfrak{B}$ the Borel $\si$-algebra
(with respect to the product topology) on $\Om$. Let $Q$ be a
probability measure on $\Om$. We assume the following about $Q$:

\begin{ass} $~$\\
\label{Assumptions}
\begin{enumerate}

\item $Q$ is stationary and ergodic with respect to each of
$\{\gth_{e_i}\}_{i=1}^{d}$, where $e_i$ is the $i^{\mbox{th}}$ principal axes and for $x\in\BZ^d$
we define $\gth_x:\Om\rightarrow\Om$ as the shift in direction $x$, i.e for every $y\in\BZ^d$ and every $\om\in\Om$
we have $\gth_{x}(\om)(y)=\om(x+y)$.\\

\item $Q(\CP(\om) = \emptyset)<1$, where
$\mathcal{P}(\om)=\{x\in\BZ^d:\om(x)=1\}$.
\end{enumerate}
\end{ass}

We denote by $\CE=\{\pm e_i\}_{i=1}^d$ the set of $2d$ points in
$\BZ^d$ with length $1$.

Let $\Om_0=\{\om\in\Om:\om(0)=1\}$, it follows from assumption
\ref{Assumptions} that $Q(\Om_0)>0$. We can therefore define the
probability $P$ on $\Om_0$ as the conditional probability on $\Om_0$
of $Q$, i.e.:
\begin{equation}
P(B)=Q(B|\Om_0)=\frac{Q(B\cap \Om_0)}{Q(\Om_0)}~~~~~\forall
B\in\mathfrak{B}. \label{P_definition}
\end{equation}
We denote by $\BE_Q$ and $\BE_P$ the expectation with respect to $Q$
and $P$ respectively.\\

\begin{clm}
Given $\om\in\Om$ and $v\in\CP(\om)$, for every vector $e\in\CE$
there exist $Q$ almost surely infinitely many $k\in\BN$ such that
$v+ke\in\CP(\om)$. \label{The_model_claim}
\end{clm}

\begin{proof}
Given $\om$,$v$ and a vector $e$ as above, since $\gth_e$ is measure
preserving and ergodic with respect to $Q$, if we define
$\Om_v=\{\om\in\Om:v\in\CP(\om)\}$ then $\mathbbm{1}_{\Om_v}\in
L^1(\Om,\mathfrak{B},Q)$, and therefore by Birkhoff's Ergodic
Theorem
\[\lim_{n\rightarrow\infty}\frac{1}{n}\sum_{k=0}^{n-1}{\gth_e^k\mathbbm{1}_{\Om_v}}=\BE_Q(\mathbbm{1}_{\Om_v})=Q(\Om_v)=Q(\Om_0)>0~~~Q~a.s.\]
Consequently, there $Q$ almost surely exist infinitely many integers
such that $\gth_e^k\mathbbm{1}_{A_v}=1$, and therefore infinitely
many $k\in\BN$ such that $v+ke\in\CP(\om)$.
\end{proof}

We define for every $v\in\BZ^d$ the set $N_v(\om)$ of the $2d$
"coordinate nearest neighbors" in $\om$, one for each direction. By
Claim \ref{The_model_claim}
$N_v(\om)$ is $Q$ almost surely a set of $2d$ points in $\BZ^d$.\\

\begin{figure}\label{fig:fig1}
\epsfig{figure=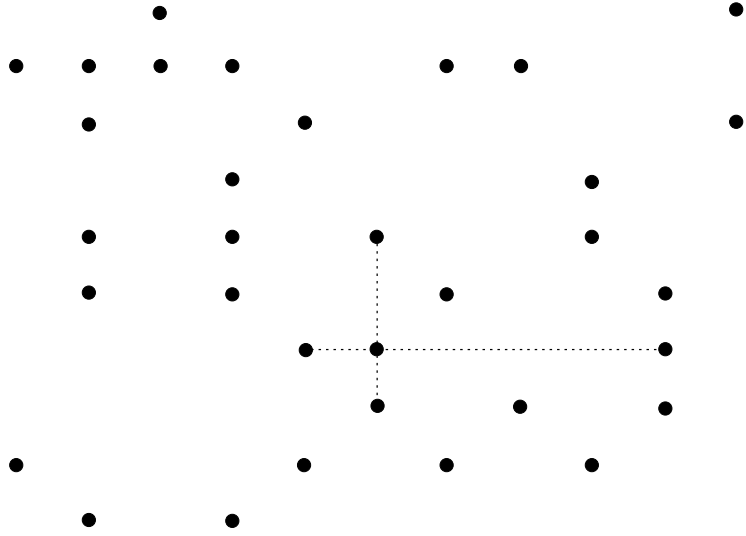, width=0.6\textwidth}
\smallskip
\caption{An example for nearest coordinate points}
\end{figure}

We can now define a random walk for $P$-almost every $\om\in \Om_0$
(on the space $((\BZ^d)^\BN,\mathcal{G},P_\om)$, where $\mathcal{G}$
is the $\si$-algebra generated by cylinder functions) as the Markov
chain taking values in $\CP(\om)$ with initial condition
\begin{equation}
P_\om(X_0=0)=1, \label{inital_condition}
\end{equation}
and transition probability
\begin{equation}
P_\om(X_{n+1}=u|X_n=v)=\left\{
\begin{array}{cc}
0&~~~u\notin N_v(\om) \\
\frac{1}{2d}&~~~u\in N_v(\om) \\
\end{array}
\right. ,\label{transition_probability}
\end{equation}
which will be called the quenched law of the random walk.
We denote the corresponding expectation by $E_\om$.

Finally, since for each $G\in\mathcal{G}$, the map
\[\om\mapsto P_\om(G),\] is $\mathcal{B}$ measurable, we may define
the probability measure $\bold{P}=P\otimes P_\om$ on
$(\Om_0\times(\BZ^d)^\BN,\mathcal{B}\times\mathcal{G})$ by
\[\bold{P}(B\times G)=\int_{B}{P_\om(G)P(d\om)},~~~~\forall B\in\mathfrak{B},~~\forall G\in\mathcal{G}.\]
The marginal of $\bf{P}$ on $(\BZ^d)^\BN$, denoted by $\BP$, is
called the annealed law of the random walk $\{X_n\}_{n=0}^\infty$.
We denote by $\BE$ the expectation with respect to $\BP$.

We will need one more definition:

\begin{defn}
For every $e\in\CE^d$ we define $f_e:\Om\rightarrow\BN^+$ by
\begin{equation}
f_e(\om)=\min\{k>0:\gth_e^k(\om)(0)=\om(ke)=1\}.
\end{equation}
\label{f_e_defn}
\end{defn}

In order to prove high dimensional Central Limit Theorem we will
assume in addition to assumption \ref{Assumptions} the following:

\begin{ass} $~$\\

(3) There exists $\ep_0>0$ such that for every coordinate direction
$e\in\CE$, $E_P(f_e^{2+\ep_0})<\infty$.
\label{assumption3}\\
\end{ass}


\subsection{Main Results}$~$\\

Our main goal is to characterize these kind of random walks on
random environments. The characterization is given by
the following theorems:\\

(1) Law of Large Numbers - For $P$ almost every $\om\in\Om_0$, the
limiting velocity of the random walk exists and equals zero. More
precisely:

\begin{thm}
Define the event
\[A=\left\{\lim_{n\rightarrow\infty}{\frac{X_n}{n}}=0\right\}.\]
Then $\BP(A)=1$. \label{LLN}
\end{thm}

(2) Recurrence Transience Classification - We give a partial
classification of recurrence transience for the random walk on a
discrete point process. The precise statements are:

\begin{prop}
The one dimensional random walk on a discrete point process is
$\BP$-almost surely recurrent. \label{tran_recu1}
\end{prop}

\begin{thm}
Let $(\Om,\FB,P)$ be a two dimensional discrete point process and
assume there exists a constant $C>0$ such that
\begin{equation}
\sum_{k=N}^{\infty}{\frac{k\cdot
P(f_{e_i}=k)}{\BE(f_{e_i})}}\leq\frac{C}{N}~~~\forall
i\in\{1,2\}~~\forall N\in\BN. \label{Cauchy_tail_ass}
\end{equation}
which in particular holds, whenever $f_{e_i}$ has a second moment
for $i\in\{1,2\}$. Then the random walk is $\BP$ almost surely
recurrent.\label{tran_recu2}
\end{thm}

\begin{thm}
Let $(\Om,\FB,P)$ be a d-dimensional discrete point process with
$d\geq 3$ then the random walk is $\BP$ almost surely
transient.\label{tran_recu3}
\end{thm}

(3) Central Limit Theorems - We prove that one-dimensional random
walks on discrete point processes satisfy a Central Limit Theorem.
We also prove that in dimension $d\geq 2$, under the additional
assumption, assumption \ref{assumption3}, the random walks on a
discrete point process satisfy a Central Limit Theorem. The precise
statements are:

\begin{thm}
Let $d=1$ and denote $e=1$ then for $P$ almost every $\om\in\Om_0$
\begin{equation}
\lim_{n\rightarrow\infty}{\frac{X_n}{\sqrt{n}}}\overset{D}{=}N(0,\BE_P^2(f_{e})).
\end{equation}
\label{CLT1}
\end{thm}

\begin{thm}
Fix $d\geq 2$. Assume the additional assumption, assumption
\ref{assumption3}, then for $P$ almost every $\om\in\Om_0$
\begin{equation}
\lim_{n\rightarrow\infty}{\frac{X_n}{\sqrt{n}}}\overset{D}{=}N(0,D),
\end{equation}
where $N(0,D)$ is a d-dimensional normal distribution with
covariance matrix $D$ that depends only on d and the distribution of
$P$.\label{CLT2}
\end{thm}

\textbf{Structure of the paper.} Sect. 2 collects some facts about
the Markov chain on environments and some ergodic results
related to it. This section is based on previously known material.
In Sect. 3 - 4 the one
dimensional case, i.e, Law of Large Numbers and Central Limit Theorem, are introduced. The Recurrence
Transience classification is discussed in Sec. 5. The novel parts of
the high dimensional Central Limit proof - asymptotic behavior of
the random walk, construction of the corrector and sublinear bounds
on the corrector - appear in Sect. 6-9. The actual proof of the high
dimensional Central Limit Theorem is carried out in Sect. 10.
Finally Sect. 11 contains further discussion, some open questions
and conjectures.


\section{The Induced shift And The Environment Seen From The Random Walk}

$~$\\
The content of this section is a standard textbook material. The
form in which it appears here is taken from \cite{BB06}. Even though
it was all known before, \cite{BB06} is the best existing source for
our purpose.

Let us define the induced shift on $\Om_0$ as follows. Let
$f_e(\om)$ be as in definition \ref{f_e_defn}. By Claim
\ref{The_model_claim} we know that $f_e(\om)<\infty$ $~Q$ almost surely
Therefore we can define the maps $\si_e:\Om_0\rightarrow\Om_0$ by
\[\si_e(\om)=\gth^{f_e(\om)}_e\om.\]
We call $\si_e$ the induced shift.

\begin{thm}
For every $e\in\CE$, the induced shift
$\si_e:\Om_0\rightarrow\Om_0$ is $P$-preserving and ergodic with
respect to $P$. \label{induced_shift_ergodicy}
\end{thm}

Theorem \ref{induced_shift_ergodicy} will follow from a more general
statement. Let $(\Delta,\mathfrak{C},\mu)$ be a probability space,
and let $T:\Delta\rightarrow\Delta$ be invertible, measure
preserving and ergodic with respect to $\mu$. Let $A\in\mathfrak{C}$
be of positive measure, and define $n:A\rightarrow\BN\cup\{\infty\}$
by
\[n(x)=\min\{k>0:T^k(x)\in A\}\]\\
The Poincar\'{e} recurrence theorem tells us that $n(x)<\infty$
almost surely. Therefore we can define, up to a set of measure zero,
the map $S:A\rightarrow A$ by
\[S(x)=T^{n(x)}(x),~~~~~x\in A\]
Then we have:

\begin{lem}
S is measure preserving and ergodic with respect to $\mu(\cdot|A)$.
It is also almost surely invertible with respect to the same
measure.\label{lem_induced_shift_ergodicy}
\end{lem}

\begin{proof}
(1) S is measure preserving: For $j\geq 1$, let $A_j=\{x\in
A:n(x)=j\}$. Then the $A_j~'s$ are disjoint and
$\mu\left(A\backslash\cup_{j\geq 1}{A_j}\right)=0$. First we show
that
\[i\neq j \Rightarrow S(A_i)\cap S(A_j)=\emptyset.\]
To do this, we use the fact that $T$ is invertible. Indeed, if $x\in
S(A_i)\cap S(A_j)$ for $1\leq i<j$, then $x=T^i(y)=T^j(z)$ for some
$y,z\in A$ with $n(y)=i,~n(z)=j$. But the fact that $T$ is
invertible implies that $y=T^{j-i}(z)$, which means $n(z)\leq
j-i<j$, a contradiction. To see that $S$ is measure preserving, we
note that the restriction of $S$ to $A_j$ is $T^j$, which is measure
preserving. Hence, $S$ is measure preserving on $A_j$ and, since the
sets $A_i$ are disjoint, $S$ is measure preserving on the union
$\cup_{j\geq
1}{A_j}$ as well.\\
(2) $S$ is almost surely invertible:
$S^{-1}(\{x\})\cap\{S~is~well~defined\}$ is a one-point set by the
fact that $T$ is itself invertible.\\
(3) $S$ is ergodic: Let $B\in\mathfrak{C}$ be such that  $B\subset
A$ and $0<\mu(B)<\mu(A)$. Assume that $B$ is $S$-invariant. Then
$S^n(x)\notin A\backslash B$ for all $x\in B$ and all $n\geq 1$.
This means that for every $x\in B$ and every $k\geq 1$ such that
$T^k(x)\in A$, we have $T^k(x)\notin A\backslash B$. It follows that
$C=\cup_{k\geq 1}{T^k(B)}$ is (almost surely) $T$-invariant and
$\mu(C)\in(0,1)$, contradicting the ergodicity of $T$.\\
\end{proof}

\begin{proof}[Proof of Theorem (\ref{induced_shift_ergodicy})]
We know that the shift $\gth_e$ is invertible, measure preserving
and ergodic with respect to $Q$. By Lemma
(\ref{lem_induced_shift_ergodicy}) the induced shift is
$P$-preserving, almost surely invertible and ergodic
with respect to $P$.\\
\end{proof}


Under the present circumstances, Theorem \ref{induced_shift_ergodicy} has one important corollary:

\begin{lem}
Let $B\in\mathcal{B}$ be a subset of $\Om_0$ such that for almost
every $\om\in B$
\begin{equation}
P_\om(\gth_{X_1}\om\in B)=1. \label{The_Lemma_0-1_Law}
\end{equation}
Then $B$ is a zero-one event under $P$. \label{lemma_sec2}
\end{lem}

\begin{proof}
The Markov property and (\ref{The_Lemma_0-1_Law}) imply that
$P_\om(\gth_{X_n}\om\in B)=1$ for all $n\geq 1$ and that $P$-almost
every $\om\in B$. We claim that $\si_e(\om)\in B$ for $P$-almost
surely $\om\in B$. Indeed, let $\om\in B$ be such that
$\gth_{X_n}\om\in B$ for all $n\geq 1$, $P_\om$-almost surely. note
that we have $f_e(\om)e\in\CP(\om)$. Therefore we have
$P_\om(X_1=f_e(\om)e)=\frac{1}{2d}>0$. This means that
$\si_e(\om)=\gth^{n(\om)}_e(\om)\in B$, i.e., $B$ is almost surely
$\si_e$-invariant. By the ergodicy of the induced shift, $B$ is a
zero-one event.
\end{proof}

Our next goal will be to prove that the Markov chain on environments
is ergodic. Let $\Xi=\Om_0^{\BZ}$ and define $\mathscr{H}$ to be the
product $\si$-algebra on $\Xi$. The space $\Xi$ is a space of
two-sided sequences - $(\ldots,\om_{-1},\om_0,\om_1,\ldots)$ - the
trajectories of the Markov chain on environments. Let $\mu$ be the
measure on $(\Xi,\mathscr{H})$ such that for any
$B\in\mathcal{B}^{2n+1}$,
\[\mu\big((\om_{-n},\ldots,\om_n)\in B\big)=\int_{B}{P(d\om_{-n})\Lambda(\om_{-n},d\om_{-n+1})\ldots\Lambda(\om_{n-1},d\om_n)},\]
where $\Lambda:\Om_0\times\mathcal{B}\rightarrow [0,1]$ is the
Markov kernel defined by
\begin{equation}
\Lambda(\om,A)=\frac{1}{2d}\sum_{x\in\BZ^d}{\big(\mathbbm{1}_{\{x\in
N_0(\om)\}}\mathbbm{1}_{\{\gth^x\om\in A\}}\big)}. \label{Lambda}
\end{equation}
Note that the sum is finite since for almost every $\om\in\Om$ there
are exactly $2d$ elements in $N_0(\om)$. $\mu$ exists and is unique by
Kolmogorov's Theorem, because $P$ is preserved by $\Lambda$, and
therefore the finite dimensional measures are consistent.
$\{\gth_{X_k}(\om)\}_{k\geq 0}$ has the same law in
$\BE_P(P_\om(\cdot))$ as $(\om_0,\om_1,\ldots)$ has in $\mu$. Let
$\widetilde{T}:\Xi\rightarrow\Xi$ be the shift defined by
$(\widetilde{T}\om)_n=\om_{n+1}$. Then $\widetilde{T}$ is measure
preserving.
\begin{prop}
$\widetilde{T}$ is ergodic with respect to $\mu$.
\label{ergodic_particle_view}
\end{prop}


\begin{proof}
Let $E_\mu$ denote expectation with respect to $\mu$. Pick
$A\subset\Xi$ that is measurable and $\widetilde{T}$-invariant. We
need to show that $\mu(A)\in\{0,1\}$.\\
Let $f:\Om_0\rightarrow\BR$ be defined as
$f(\om_0)=E_\mu(\mathbbm{1}_A|\om_0).$ First we claim that
$f=\mathbbm{1}_A$ almost surely. Indeed, since $A$ is
$\widetilde{T}$-invariant, there exist $A_+\in\si(\om_k:k>0)$ and
$A_-\in\si(\om_k:k<0)$ such that $A$ and $A_\pm$ differ only by null
sets from one another (This follows by approximation of $A$ by
finite-dimensional events and using the $\widetilde{T}$-invariance
of $A$). Now, conditional on $\om_0$, the event $A_+$ is independent
of $\si(\om_k:k<0)$ and so L$\acute{e}$vy's Martingale Convergence
Theorem gives us
\[E_\mu(\mathbbm{1}_A|\om_0)=E_\mu(\mathbbm{1}_{A_+}|\om_0,\om_{-1},\ldots,\om_{-n})\]
\[=E_\mu(\mathbbm{1}_{A_-}|\om_0,\ldots,\om_{-n})~\overrightarrow{_{n\rightarrow\infty}}~\mathbbm{1}_{A_-}=\mathbbm{1}_A,\]
with equalities valid $\mu$-almost surely. Next let $B\subset\Om_0$
be defined by $B=\{\om_0:f(\om_0)=1\}$. Clearly $B$ is
$\mathfrak{B}$-measurable and, since the $\om_0$-marginal of $\mu$
is $P$, \[\mu(A)=E_\mu(f)=P(B)\] Hence, in order to prove that
$\mu(A)\in\{0,1\}$, we need to show that $P(B)\in\{0,1\}$ But $A$ is
$\widetilde{T}$-invariant and so, up to set of measure zero, if
$\om_0\in B$ then $\om_1\in B$. This means that $B$ satisfies the
condition of the lemma \ref{lemma_sec2}, and so $B$ is a zero-one event.\\
\end{proof}

\begin{thm}
Let $f\in L^1(\Om_0,\mathcal{B},P)$. Then for $P$-almost all
$\om\in\Om_0$
\[\lim_{n\rightarrow\infty}{\frac{1}{n}\sum_{k=0}^{n-1}{f\circ\gth_{X_k}(\om)}=\BE_P(f)}~~~~~~P_\om~almost~surely.\]
Similarly, if $f:\Om\times\Om\rightarrow\BR$ is measurable with
$\BE_P(E_\om(f(\om,\gth_{X_1}\om)))<\infty$, then
\[\lim_{n\rightarrow\infty}{\frac{1}{n}\sum_{k=0}^{n-1}{f(\gth_{X_k}\om,\gth_{X_{k+1}}\om)}=\BE_P(E_\om(f(\om,\gth_{X_1}\om)))}.\]
for $P$-almost all $\om$ and $P_\om$-almost all trajectories of
$(X_k)_{k\geq 0}$. \label{Thm_mutual_ergodic}
\end{thm}

\begin{proof}
Recall that $\{\gth_{X_k}(\om)\}_{k\geq 0}$ has the same law in
$\BE_P(P_\om(\cdot))$ as $(\om_0,\om_1,\ldots)$ has in $\mu$. Hence,
if $g(\ldots,\om_{-1},\om_0,\om_1,\ldots)=f(\om_0)$ then
\[\lim_{n\rightarrow\infty}{\frac {1}{n}\sum_{k=0}^{n-1}{f\circ\gth_{X_k}}}\overset{D}{=}\lim_{n\rightarrow\infty}{\frac{1}{n}\sum_{k=0}^{n-1}{g\circ
\widetilde{T}^k}}.\] The latter limit exists by Birkhoff's Ergodic
Theorem (we have already seen that $\widetilde{T}$ is ergodic) and
equals $E_\mu(g)=\BE_P(f)$ almost surely. The second part is proved
analogously.
\end{proof}


\section{Law of Large Numbers}

We turn now to prove Theorem \ref{LLN} - i.e. Law of Large Numbers
for random walks on a discrete point process. For completeness we
state the theorem again:

\begin{theo}\textsl{\textbf{\ref{LLN}}}
Define the event
\[A=\left\{\lim_{n\rightarrow\infty}{\frac{X_n}{n}}=0\right\}.\]
Then $\BP(A)=1.$
\end{theo}

\begin{proof}
Using linearity, it is enough to prove that for every  $e\in\CE$
we have $\BP(A_e)=1$, where
\[A_e=\left\{\lim_{n\rightarrow\infty}{\frac{X_n\cdot e}{n}}=0\right\}.\]
For every $e\in\CE$ let $f_e$ be as in
Definition \ref{f_e_defn}. By (\ref{The_model_claim}) $f_e$ is $P$-a.s
finite. We first prove that $\BE_P(f_e)<\infty$. Assume for contradiction that $\BE_P(f_e)=\infty$, since
$f_e$ is positive then
\begin{equation}
\lim_{n\rightarrow\infty}{\frac{1}{n}\sum_{k=0}^{n-1}{f_e(\si_e^k(\om))}}=\infty~~~~~P~a.s.
\label{lim_f}
\end{equation}
Indeed, for every $M>0$ define
\[f_e^M(\om)=
\left\{
\begin{array}{cc}
 f_e(\om) & f_e(\om)\leq M \\
 M      & f_e(\om)>M  \\
\end{array}
\right.,\]
then, since $f_e^M$ is finite by Birkhoff Ergodic Theorem
\[\lim_{n\rightarrow\infty}{\frac{1}{n}\sum_{k=0}^{n-1}{f_e(\si_e^k(\om))}}\geq\lim_{n\rightarrow\infty}{\frac{1}{n}\sum_{k=0}^{n-1}{f_e^M(\si_e^k(\om))}}=\BE_P(f_e^M),~~~~~P~a.s.\]
Taking now M to infinity we get
\[\lim_{n\rightarrow\infty}{\frac{1}{n}\sum_{k=0}^{n-1}{f_e(\si_e^k(\om))}}\geq\lim_{M\rightarrow\infty}{\BE_P(f_e^M)}=\BE_P(f_e)=\infty,~~~~~P~a.s.\]
Let $S(k)=max\left\{n\geq
0:\sum_{m=0}^{n-1}{f(\si_e^m(\om))}<k\right\}$, by (\ref{lim_f})
\[\lim_{k\rightarrow\infty}{\frac{S(k)}{k}}=0~~~~~P~a.s.\]\\
On the other hand, let $g:\Om\rightarrow\{0,1\}$ be defined by
\[g(\om)=\mathbbm{1}_{\Om_0}(\om),\]
then
\[S(k)=\sum_{j=0}^{k-1}{g(\gth_e^j(\om))},\]
and therefore by Birkhoff's Ergodic Theorem
\[Q(\Om_0)=\lim_{k\rightarrow\infty}{\frac{1}{k}\sum_{j=0}^{k-1}{g(\gth_e^j(\om))}}=\lim_{k\rightarrow\infty}{\frac{S(k)}{k}}=0,~~~~~P~a.s,\]
contradicting assumption~\ref{Assumptions}. It follow that
$\BE_P(f_e)<\infty$, and therefore by Birkhoff Ergodic Theorem
\begin{equation}
\lim_{n\rightarrow\infty}{\frac{1}{n}\sum_{k=0}^{n-1}{f_e(\si_e^k(\om))}}=\BE_P(f_e)<\infty~~~~~P~a.s.\label{Finite_sum_of_changes}
\end{equation}

Notice that
\begin{equation}
P(f_{-e}(\om)=k)=P(f_e(\si_e^{-1}(\om))=k)=P(f_e(\om)=k),\label{Symmetry}
\end{equation}
where the last equality is true since $P$ is stationary. It therefore follows that
\begin{equation}
\BE_P(f_e)=\BE_P(f_{-e}) \label{Expectation_equality}
\end{equation}

For $e\in\CE$ let $g_e:\Om\times\Om\rightarrow\BZ$ be as follows:
\[g_e(\om,\om')= \left\{
\begin{array}{cc}
f_e(\om) & \om'=\si_e(\om) \\
-f_{-e}(\om) & \om'=\si_{-e}(\om) \\
0 & otherwise \end{array}\right..\]
Now, $g_e$ is measurable and
using (\ref{Expectation_equality}) we get
\[\BE_P(E_\om(g_e(\om,\gth^{X_1}\om)))=\BE_P\left(\frac{1}{2d}f_e(\om)-\frac{1}{2d}f_{-e}(\om)\right)=0.\]
It therefore follows that for every $e\in\CE$,
for almost every $\om\in\Om_0$ and $P_\om$ almost every random walk $\{X_k\}_{k\geq
0}$, we have for $Z_k=X_k-X_{k-1},~~k\geq 1$ that
\[\lim_{n\rightarrow\infty}{\frac{X_n\cdot e}{n}}=\lim_{n\rightarrow\infty}{\frac{1}{n}\sum_{k=1}^{n}{Z_k\cdot e}},\]
and from \eqref{Thm_mutual_ergodic} this equals to
\[\lim_{n\rightarrow\infty}{\frac{1}{n}\sum_{k=0}^{n-1}{g_e(\gth_{X_k}\om,\gth_{X_{k+1}}\om)}}=\BE_P(E_\om(g_e(\om,\gth_{X_1}\om)))=0,~~~\BP~a.s.\]
\end{proof}


\section{One Dimensional Central Limit Theorem}

Here we prove Theorem \ref{CLT1} - i.e. Central Limit Theorem for
one dimensional random walks on discrete point processes. We start
by stating the theorem

\begin{theo}\textsl{\textbf{\ref{CLT1}}}
Let $d=1$ and denote $e=1$ then for $P$ almost every $\om\in\Om_0$
\begin{equation}
\lim_{n\rightarrow\infty}{\frac{X_n}{\sqrt{n}}}\overset{D}{=}N(0,\BE_P^2(f_{e})).
\end{equation}
\end{theo}

\begin{proof} 
We first notice that for $d=1$, a random walk on a discrete point
process is almost surely a simple one dimensional random walk with
changed distances between points. Secondly the expectation of the
distance between points, given by $\BE_P(f_{e})$, is finite.

Given an environment $\om\in\Om_0$ and a random walk $\{X_k\}_{k\geq
0}$, we define the simple one-dimensional random walk
$\{Y_k\}_{k\geq 0}$ associated with $\{X_k\}_{k\geq 0}$ as follows:
First we define $Z_k=X_k-X_{k-1}$ for every $k\geq 1$, then we
define $W_k=\frac{Z_k}{|Z_k|}$. Finally we define $Y_0=0$ and for
$k\geq 1$ we define $Y_k=\sum_{j=1}^{k}{W_k}$. Since $\{Y_k\}_{k\geq
0}$ is a simple one dimensional random walk on $\BZ$, it follows
from the Central Limit Theorem that for $P$ almost every
$\om\in\Om_0$
\begin{equation}
\lim_{n\rightarrow\infty}{\frac{1}{\sqrt{n}}\cdot
Y_n}\overset{D}{=}N(0,1).
\end{equation}

We now turn to define for every $\om\in\Om$ the points of the
environment. For every $n\in\BZ$ let $t_n$ be the $n^{\mbox{th}}$
place on the grid with a point, i.e. $t_0=0$,
\[t_n=\sum_{k=0}^{n-1}{f_{e}(\si_e^k\om)}~~~~~n>0,\]
and
\[t_n=\sum_{k=-1}^{-n}{f_{e}(\si_e^k\om)}~~~~~n<0.\]

For every $a>0$ we have
\[
\lim_{n\rightarrow\infty}{\frac{1}{\sqrt{n}}t_{\lfloor a\sqrt{n}\rfloor}}
=a\cdot\lim_{n\rightarrow\infty}{\frac{1}{a\sqrt{n}}\sum_{k=0}^{\lfloor
a\sqrt{n}\rfloor}{f_{e}(\si_e^k\om)}}=a\cdot \BE_P(f_{e}),
\]
where the last equality holds since this sequence contain the same
elements as the sequence in (\ref{Finite_sum_of_changes}) and every
element in the original sequence appears only a finite number of
times, therefore those sequences have the same partial limits, and
the original sequence (the one in \eqref{Finite_sum_of_changes}) converges.

By the same argument for every $a\in\BR$ we have that
\begin{equation}
\lim_{n\rightarrow\infty}{\frac{1}{\sqrt{n}}t_{\lfloor
a\sqrt{n}\rfloor}}=a\cdot \BE_P(f_{e}). \label{limit_CLT_constant}
\end{equation}

Using (\ref{limit_CLT_constant}) and the fact that
$\lim_{n\rightarrow\infty}{\frac{Y_n}{\sqrt{n}}}$ exists and finite
$\BP$ almost surely, we get that
\[\lim_{n\rightarrow\infty}{\frac{1}{\sqrt{n}}t_{Y_n}}=\lim_{n\rightarrow\infty}{\frac{1}{\sqrt{n}}t_{\frac{Y_n}{\sqrt{n}}\sqrt{n}}}\leq
a~~\Leftrightarrow~~\lim_{n\rightarrow\infty}{\frac{Y_n}{\sqrt{n}}\leq\frac{a}{\BE_P(f_{e})}},\]
and therefore
\[\BP\left(\lim_{n\rightarrow\infty}{\frac{1}{\sqrt{n}}t_{Y_n}}\leq a\right)
=\BP\left(\lim_{n\rightarrow\infty}{\frac{Y_n}{\sqrt{n}}\leq\frac{a}{\BE_P(f_{e})}}\right)=\Phi\left(\frac{a}{\BE_P(f_{e})}\right),\]
where $\Phi$ is the standard normal cumulative distribution
function. Finally, we notice that
\[X_n=t_{Y_n},\]
and therefore we conclude that
\[\BP\left(\lim_{n\rightarrow\infty}{\frac{X_n}{\sqrt{n}}}\leq
a\right)=\Phi\left(\frac{a}{\BE_P(f_{e})}\right),\] as required.
\end{proof}


\section{Transience and Recurrence}

Before we continue the discussion on Central Limit Theorem in higher
dimensions, we turn to deal with transience and recurrence of random
walks on discrete point processes.

\subsection{One-dimensional case}

\begin{propo}\textsl{\textbf{\ref{tran_recu1}}}
The one dimensional random walk on a discrete point process is
$\BP$-almost surely recurrent.
\end{propo}

\begin{proof}[Proof of Proposition \ref{tran_recu1}]
Using the notation from the previous section, since $Y_n$ is a
one-dimensional simple random walk, it is recurrent $\BP$ almost
surely. Therefore we have $\#\{n:Y_n=0\}=\infty~~\BP$ almost surely, but since
$X_n=t_{Y_n}$ and $t_0=0$ we have $\#\{n:X_n=0\}=\infty~~\BP$ almost surely,
and therefore the random walk is recurrent.
\end{proof}

\subsection{Two-dimensional case}

The theorem we wish to prove is the following:

\begin{theo}\textsl{\textbf{\ref{tran_recu2}}}
Let $(\Om,\FB,P)$ be a two dimensional discrete point process and
assume there exists a constant $C>0$ such that
\begin{equation}
\sum_{k=N}^{\infty}{\frac{k\cdot
P(f_{e_i}=k)}{\BE(f_{e_i})}}\leq\frac{C}{N},\quad\forall
i\in\{1,2\}~~\forall N\in\BN, \label{Cauchy_tail_ass}
\end{equation}
which in particular holds, whenever $f_{e_i}$ has a second moment
for $i\in\{1,2\}$. Then the random walk is $\BP$ almost surely
recurrent.
\end{theo}

The proof is based on the connection between random walks,
electrical networks and the Nash-William criteria for recurrence of
random walks. For a proof of the Nash-William criteria and some
background on the subject see \cite{DS84} and \cite{LP97}.

We start with the following definition:
\begin{defn}
Let $(\widetilde{\Om},\widetilde{\mathfrak{B}},\widetilde{P})$ be a
probability space. We say that a random variable
$X:\widetilde{\Om}\rightarrow [0,\infty)$ has a Cauchy tail if there
exist a positive constant $C$ such that for every $n\in \BN$ we have
\begin{equation}
\widetilde{P}\left(X\geq n\right)\leq \frac{C}{n}. \nonumber
\end{equation}
Note that if $\widetilde{E}(X)<\infty$, then X has a Cauchy
tail.
\end{defn}

In order to prove theorem \ref{tran_recu2} we will need the
following lemmas taken from \cite{Be01}.

\begin{lem}[\cite{Be01} Lemma 4.1]
Let $\{f_i\}_{i=1}^{\infty}$ be identically distributed positive
random variables, on a probability space
$(\widetilde{\Om},\widetilde{\mathfrak{B}},\widetilde{P})$, that
have a Cauchy tail. Then, for every $\epsilon>0$, there exist $K>0$ and $N\in\BN$

such that for every $n>N$
\[\widetilde{P}\left(\frac{1}{n}\sum_{k=0}^{n}{f_i}>K\log n\right)<\epsilon.\]
\label{Cauchy_tail_lemma}
\end{lem}

\begin{proof}
$f_i$ has a Cauchy tail, so there exists $C_0$ such that for every $n\in\BN$
\[\widetilde{P}(f_i>n)<\frac{C_0}{n}\]
Let $M>\frac{2}{\ep}$ be a large number, and $N$ large enough that
$C_0N^{1-M}<\frac{\ep}{2}$. Fix $n>N$, and let $g_i=\min\{f_i,n^M\}$
for all $1\leq i\leq n$. Then,
\[\widetilde{P}\left(\frac{1}{n}\sum_{i=1}^{n}{f_i}\neq\frac{1}{n}\sum_{i=1}^{n}{g_i}\right)\leq\sum_{k=1}^{n}{\widetilde{P}(f_k\neq g_k)}=n\cdot \widetilde{P}(f_1\neq g_1).\]
The last term is equal to
\[n\cdot \widetilde{P}(f_1>n^M)<\frac{n\cdot C_0}{n^M}<\frac{\ep}{2}.\]
Now, since $E(g_i)\leq C_0M \log n$, and $g_i$ is positive, by
Markov's inequality, choosing $K=C_0M^2$ we get
\[\widetilde{P}\left(\frac{1}{n}\sum_{i=1}^{n}{g_i}>K\log n\right)<\frac{C_0M\log n}{C_0M^2\log n}=\frac{1}{M}<\frac{\ep}{2}\]
and so
\[\widetilde{P}\left(\frac{1}{n}\sum_{i=1}^{n}{f_i}>K\log n\right)<\ep\]\\
\end{proof}

\begin{lem}[\cite{Be01} Lemma 4.2]
Let $A_n$ be a sequence of events such that $\tilde{P}(A_n)>1-\epsilon$ for
all sufficiently large $n$, and let $\{a_n\}_{n=1}^\infty$ be a sequence such
that
\[\sum_{n=1}^{\infty}{a_n}=\infty.\]
Then, with probability of at least $1-\ep$
\[\sum_{n=1}^{\infty}{\mathbbm{1}_{A_n}\cdot a_n}=\infty.\]
\label{lem2_2dim_recurrent}
\end{lem}

\begin{proof}
It is enough to show that there exists $N$ such that for any $M$,
\begin{equation}
P\left(\sum_{n=N}^{\infty}{\mathbbm{1}_{A_n}\cdot
a_n}<M\right)\leq\ep. \label{lem_claim}
\end{equation}
Define $N$ such that for every $n>N$ we have $P(A_n)>1-\epsilon$,
and assume that for some $M$ (\ref{lem_claim}) is false. Define
$B_M$ to be the event
\[B_M=\left\{\sum_{n=N}^{\infty}{\mathbbm{1}_{A_n}\cdot a_n}<M\right\}.\]
Since $P(B_M)>\ep$, we know that there exist $\delta>0$ such that
for every $n$
\[P(A_n|B_M)=\frac{P(A_n\cap B_M)}{P(B_M)}\geq\frac{P(B_M)-\ep}{P(B_M)}>\delta>0.\]
Therefore,
\[E\left[\sum_{n=N}^{\infty}{\mathbbm{1}_{A_n}\cdot a_n}\Bigg|B_M\right]\geq\delta\sum_{n=N}^{\infty}{a_n}=\infty,\]
which contradicts the definition of $B_M$.\\
\end{proof}

In the proof of Theorem \ref{tran_recu2} we will use the following
notation: Given a graph $G=(V,E)$ with $V\subset\BZ^d$, for every
$e\in E$ define $e^+\in V$ and $e^-\in V$ to be the end points of
$e$, such that if $(e^+-e^-)\cdot e_i\neq 0$ then $(e^+-e^-)\cdot
e_i>0$. In addition for every $e\in E$ we write $l(e)=|e^+-e^-|_1$.

\begin{proof}[Proof of theorem \ref{tran_recu2}]
For every $\om\in\Om$, we define the corresponding network with
conductances $G(\om)=(V(\om),E(\om),c(\om))$ as follows: First let
$G''(\om)=(V''(\om),E''(\om),c''(\om))$ be the network with
$V''(\om)=\CP(\om)$ and $E''(\om)=\{\{x,y\}\in V''\times
V'':y\in\{x\pm f_{e_1}(\om)e_1,x\pm f_{e_2}(\om)e_2\}\}$, i.e. the
set of edges from each point to its four "nearest neighbors". we
also define the conductance $c''(\om)(e)=1$ for every $e\in
E''(\om)$. We now define $G'(\om)$ to be the network generated from $G''(\om)$ by "cutting" every
edge of length $k$ into $k$ edges of length $1$, each cut with conductance
$k$. Formally we define $V'(\om)=V^1(\om)\biguplus
V^2(\om)\subset\BZ^2\times\{0,1\}$ where

\[V^i(\om)=\left\{(x,i):\begin{array}{cc} \exists~e\in E''(\om)
~\exists~0\leq k\leq l(e)\\~such~that~(e^+-e^-)\cdot e_i\neq
0~\wedge~x=e^-+ke_i\end{array}\right\},\]

and we define $E'(\om)=E^1(\om)\cup E^2(\om)$ by

\[E^i(\om)=\left\{\left\{(v,i),(w,i)\right\}:\begin{array}{cc} \exists~e\in E''(\om)~\exists~0\leq k<
l(e)~~such~that~\\~(e^+-e^-)\cdot e_i\neq 0~\wedge
v=e^-+ke_i,w=e^-+(k+1)e_i\end{array}\right\}.\]
We also define the
conductance $c'(\om)(e)$ of an edge $e\in E'(\om)$ to be $k$, given
that the length of the original edge it was part of was $k$. Finally
we define $G(\om)$ to be the graph generated from $G'(\om)$ by
identifying every $v\in V''(\om)$ on both levels i.e, we take the
graph $G'(\om)$ modulo the equivalence relations
$(v,1)=(v,2)~~\forall~v\in V''(\om)$. We now turn to prove the
recurrence using the Nash-Williams Criteria. Let $\Pi_n$ be the set
of edges exiting the box $([-n,n]\times [-n,n],[1,2])$ in the graph
$G(\om)$. Then $\Pi_n$ defines a sequence of pairwise disjoint
cutsets in the network $G(\om)$. Let $e\in\Pi_n$ be such that
$(e^+-e^-)\cdot e_i\neq 0$ then
\[
\BP(c(e)=k)=\BP\left(\begin{array}{cc}\text{the original edge that contained}\\
e \text{ is of length } k\end{array}\right)=\frac{k\cdot
\BP(f_{e_i}=k)}{\BE(f_{e_i})}.
\]
Indeed, the probability that the edge e was part of an edge of
length $k$ in the original graph, needs to be multiplied by $k$, since
it can be in any part of the edge. From assumption \eqref{Cauchy_tail_ass} it follows that $c(e)$ has a
Cauchy tail. In $\Pi_n$ there are $2n+4$ edges in the first level
and $2n+4$ in the second level, all of them with the same
distribution (and by \eqref{Cauchy_tail_ass} a Cauchy tail), though
they may be dependent. By Lemma \ref{Cauchy_tail_lemma}, for every
$\ep>0$ there exist $K>0$ and $N\in\BN$ such that for every $n>N$, we have
\begin{equation}
\BP\left(\sum_{e\in\Pi_n}{C(e)}\leq K(4n+8)\log 4n+8\right)>1-\ep.
\label{Cauchy_event}
\end{equation}

Define $A_n$ to be the event in equation (\ref{Cauchy_event}), and set
$a_n=(K(4n+8)\log (4n+8))^{-1}$ for $n\geq N$. Now,
\[\sum_{n=1}^{\infty}{{C_{\Pi_n}}^{-1}}\geq
\sum_{n=N}^{\infty}{\mathbbm{1}_{A_n}\cdot a_n}.\] By the definition
of $\{a_n\}$,
\[\sum_{n=N}^{\infty}{a_n}=\infty.\]
On the other hand, $\BP(A_n)>1-\ep$ for all n. So by Lemma
\ref{lem2_2dim_recurrent},
\[\BP\left(\sum_{n=1}^{\infty}{{C_{\Pi_n}}^{-1}}=\infty\right)\geq 1-\ep.\]
Since $\ep$ is arbitrary, we get that $\BP$ a.s.
\[\sum_{n=1}^{\infty}{{C_{\Pi_n}}^{-1}}=\infty.\]
Therefore by the Nash-Williams criteria, the random walk is $\BP$
almost surely recurrent on $G(\om)$.
\end{proof}

\subsection{Higher dimensions ($d\geq 3$)}

$~$\\

We start by stating the theorem:

\begin{theo}\textsl{\textbf{\ref{tran_recu3}}}
Let $(\Om,\FB,P)$ be a $d$-dimensional discrete point process $d\geq
3$ then the random walk is $\BP$ almost surely transient.\\
\end{theo}

The main idea beyond the proof is as follows: first we show that the
boundary of every set of volume $n$ in $\BZ^d$ is at least a positive
constant times $n^{\frac{d-1}{d}}$, then we will use the known fact
that for every graph $G=(V,E)$ with bounded degree, such that for
every set of vertices of volume $n$, the boundary is at least a
constant times $n^\alpha$, with $\alpha>\frac{1}{2}$, a simple
random walk on G is transient.

We start by proving an isoperimetric inequality.

\begin{lem}
Let $A=\{x^i=(x_1^i,x_2^i,\ldots,x_d^i)\}_{i=1}^n$ be a finite subset of
$\BZ^d$. We define $\Pi^j:\BZ^d\rightarrow\BZ^{d-1}$
to be the projection on all but the $j^{th}$ coordinate, i.e,
$\Pi^j(x)=\Pi^j((x_1,x_2,\ldots,x_d))=(x_1,x_2,\ldots,x_{j-1},x_{j+1},\ldots,x_d)$.
Define $A_j=\Pi^j(A)$. Then there exists $C>0$ such that
\begin{equation}
\max_{1\leq j\leq d}\{|A_j|\}\geq C\cdot |A|^{\frac{d-1}{d}},
\end{equation}
where $|\cdot|$ denotes the cardinality of the set.
\label{lem_volume_surface}
\end{lem}

\begin{proof}
Using translation, we can assume without loss of generality that
$x_j^i>0$ for every $1\leq i\leq n$ and $1\leq j \leq d$. For every
point $x$ in the quadrat, where all coordinates are positive, we define the energy of a point $\mathcal{E}(x)$ by
\begin{equation}
\mathcal{E}(x)=x\cdot (1,1,\ldots,1)=\sum_{j=1}^d x_j.
\end{equation}
In addition we define the energy of a finite set A in this quadrat as
\begin{equation}
\mathcal{E}(A)=\sum_{x\in A}{\mathcal{E}(x)}.
\end{equation}
For each point $(x_2,x_3,\ldots,x_d)$ in $\BZ^{d-1}$ with positive
entries we define the set
$A_{(x_2,x_3,\ldots,x_d)}=\{x_1:(x_1,x_2,\ldots,x_d)\in A\}$, which
we will call the $(x_2,x_3,\ldots,x_d)$ fiber of $A$. We now define
a new set $A^1$, with the following property: For each point
$(x_2,x_3,\ldots,x_d)$ in $\BZ^{d-1}$ the $(x_2,x_3,\ldots,x_d)$
fiber of $A$ as the same size as the $(x_2,x_3,\ldots,x_d)$ fiber of
$A^1$, and in addition the $(x_2,x_3,\ldots,x_d)$ fiber of $A^1$ is
the one with least energy (when thought as a set in $\BZ$). We claim
that the following set fulfills this property:
\begin{equation}
A^1=\bigcup_{x_2\in\BN}\bigcup_{x_3\in\BN}\ldots\bigcup_{x_d\in\BN}{\{(a,x_2,x_3,\ldots,x_d):a\in\BN~\wedge~1\leq
a\leq |A_{(x_2,x_3,\ldots,x_d)}|\}}.
\end{equation}

Indeed, the $(x_2,x_3,\ldots,x_d)$ fiber of $A^1$ is
$\{(a,x_2,x_3,\ldots,x_d):a\in\BN~\wedge~1\leq a\leq
|A_{(x_2,x_3,\ldots,x_d)}|\}$ which has the same size as the
$(x_2,x_3,\ldots,x_d)$ fiber of $A$. In addition, for any fixed
$m\in\BN$, the unique set $B\subset\BN$ of size $m$ and minimal energy is $B=\{1,2,\ldots,m\}$.

Therefore the set $A^1$ has the following properties:
\begin{enumerate}

\item $|A^1|=n$.\\

\item $|\Pi^j(A^1)|\leq |\Pi^j(A)|$ for every $1\leq j \leq d$.\\

\item $\mathcal{E}(A^1)\leq\mathcal{E}(A)$, and equality holds if and only if $A^1=A$. \\

\end{enumerate}

Indeed,
\begin{enumerate}

\item This follows from the fact that the size of the fibers don't change in the process, and that the fibers are disjoint.
\begin{equation}
|A|=\sum_{x_2\in\BN}{\sum_{x_3\in\BN}{\ldots\sum_{x_d\in\BN}{|A_{(x_2,x_3,\ldots,x_d)}|}}}
=\sum_{x_2\in\BN}{\sum_{x_3\in\BN}{\ldots\sum_{x_d\in\BN}{|A^1_{(x_2,x_3,\ldots,x_d)}|}}}
=|A^1|. \nonumber
\end{equation}

\item For $j=1$ this is true since
\begin{align}
(x_2,x_3,\ldots,x_d)\in\Pi^1(A)  \Leftrightarrow  \exists ~
a\in\BN~such~that~(a,x_2,x_3,\ldots,x_d)\in A  \Leftrightarrow \nonumber\\
\exists ~b\in\BN ~such~that~ (b,x_2,x_3,\ldots,x_d)\in A^1
 \Leftrightarrow (x_2,x_3,\ldots,x_d)\in\Pi^1(A^1),  ~\nonumber
\end{align}
and therefore  $|\Pi^1(A^1)|=|\Pi^1(A)|$.
For $2\leq j\leq d$, we assume for contradiction
that, $|\Pi^j(A^1)|> |\Pi^j(A)|$. Then there exist
$(x_1,x_2,\ldots,x_{j-1},x_{j+1},\ldots,x_d)\in \Pi^j(A^1)$ such
that
\begin{equation}
|\Pi^j(A^1)_{(x_2,\ldots,x_{j-1},x_{j+1},\ldots,x_d)}|>
|\Pi^j(A)_{(x_2,\ldots,x_{j-1},x_{j+1},\ldots,x_d)}|. \nonumber
\end{equation}
From the definition of $A^1$ there exists $m\in\BN$ such that
\begin{equation}
|\Pi^j(A^1)_{(x_2,\ldots,x_{j-1},x_{j+1},\ldots,x_d)}|=|A^1_{(x_2,\ldots,x_{j-1},m,x_{j+1},\ldots,x_d)}|,
\nonumber
\end{equation}
and since for every $k\in\BN$
\begin{equation}
|A_{(x_2,\ldots,x_{j-1},k,x_{j+1},\ldots,x_d)}|\leq
|\Pi^j(A)_{(x_2,\ldots,x_{j-1},x_{j+1},\ldots,x_d)}|. \nonumber
\end{equation}
It follows that
\begin{equation}
|A^1_{(x_2,\ldots,x_{j-1},m,x_{j+1},\ldots,x_d)}|>
|A_{(x_2,\ldots,x_{j-1},m,x_{j+1},\ldots,x_d)}|,
\end{equation}
which contradicts the fact that the size of fibers in $A$ and $A^1$
is the same.\\

\item By definition
\begin{equation}
\begin{aligned}
\mathcal{E}(A^1)&=\sum_{x_2\in\BN}{\sum_{x_3\in\BN}{\ldots\sum_{x_d\in\BN}{~~\sum_{x_1\in
A^1_{(x_2,x_3,\ldots,x_d)}}{(x_1+x_2+\ldots+x_d)}}}} \nonumber\\
&=\sum_{x_2\in\BN}{\sum_{x_3\in\BN}{\ldots\sum_{x_d\in\BN}{\left[|A^1_{(x_2,x_3,\ldots,x_d)}|(x_2+x_3+\ldots+x_d)+\mathcal{E}(A^1_{(x_2,x_3,\ldots,x_d)})\right]}}}
\nonumber\\
&=\sum_{x_2\in\BN}{\sum_{x_3\in\BN}{\ldots\sum_{x_d\in\BN}{\left[|A_{(x_2,x_3,\ldots,x_d)}|(x_2+x_3+\ldots+x_d)+\mathcal{E}(A^1_{(x_2,x_3,\ldots,x_d)})\right]}}}
\nonumber\\
&\leq\sum_{x_2\in\BN}{\sum_{x_3\in\BN}{\ldots\sum_{x_d\in\BN}{\left[|A_{(x_2,x_3,\ldots,x_d)}|(x_2+x_3+\ldots+x_d)+\mathcal{E}(A_{(x_2,x_3,\ldots,x_d)})\right]}}}
\nonumber\\
&=\mathcal{E}(A) \nonumber\\
\end{aligned}
\end{equation}
where the inequality is true since the energy of the
$(x_2,x_3,\ldots,x_d)$ fiber of $A^1$ is the one with least energy from all $(x_2,x_3,\ldots,x_d)$ fibers of $A$. In addition equality
holds if and only if for every $(x_2,x_3,\ldots,x_d)$ fiber of $A$
we have
$\mathcal{E}(A_{(x_2,x_3,\ldots,x_d)})=\mathcal{E}(A^1_{(x_2,x_3,\ldots,x_d)})$
which is possible if and only if
$A_{(x_2,x_3,\ldots,x_d)}=A^1_{(x_2,x_3,\ldots,x_d)}$, since $A^1$
fibers were chosen to be with minimal energy.
\end{enumerate}

Repeating the last procedure for the set $A^i$ with the $i+1^{\mbox{th}}$
coordinate instead of the first one we obtain the sets
$A^2,\ldots,A^d$, with the same number of point, decreasing
energy and decreasing size of projections. Let
$\widetilde{A}_0\equiv A$, and define by induction
$\widetilde{A}_{n+1}=\widetilde{A}_n^d$ be the set generated from
$\widetilde{A}_n$ by repeating the last procedure. It follows that
sequence of sets $\{\widetilde{A}_n\}_{n=0}^{\infty}$ contains only
finite number of sets. Indeed since the energy of a set is an
natural number, and the energy can only decrease as $n$ increases,
there exist N such that for every $n\geq N$ the energy is constant.
Using  now property (3) it follows that
$\widetilde{A}_n=\widetilde{A}_{n+1}$ for every $n\geq N$ and
therefore there is only finite number of sets in the sequence. Let
$\widehat{A}$ be the limiting set of the sequence. Note that the boundary of
$\widehat{A}$ is exactly
$2\sum_{i=1}^{d}{\Pi^i\big(\widehat{A}\big)}$, because otherwise one
can decrease the energy. Using the fact that the boundary of every
set of size $n$ in $d$ dimensions is at least $C_0\cdot
n^{\frac{d-1}{d}}$ for some positive constant $C_0$, see
\cite{DP96}, we get that there exist a positive constant $C$ and at
least one $i_0\in\{1,2,\ldots,d\}$ such that
$\Pi^{i_0}\big(\widehat{A}\big)\geq C\cdot n^{\frac{d-1}{d}}$, and
therefore $\Pi^{i_0}(A)\geq C_0n^{\frac{d-1}{d}}$ for the original
set $A$ too, as required.
\end{proof}

We now turn to define the isoperimetric profile of a graph. Let
$\{p(x,y)\}_{x,y\in V}$ be transition probabilities for an irreducible Markov
chain on a countable state space V (we will think about this Markov
chain as the random walk on a weighted graph $G=(V,E,C)~$, with
$\{x,y\}\in E$ if and only if $p(x,y)>0$ and for every $\{x,y\}\in
E$ we define the conductance $C(x,y)=p(x,y)$. For $S\subset V$, the
"boundary size" of S is measured by $|\partial S|=\sum_{s\in
S}{\sum_{a\in S^c}{p(s,a)}}$. We define $\Phi_S$, the conductance of
S, by $\Phi_S:=\frac{|\partial S|}{|S|}$. Finally, define the
isoperimetric profile of the graph G,  with vertices V and
conductances induced from the transition probabilities by:
\begin{equation}
\Phi(u)=\inf\{\Phi_S:S\subset V,~|S|\leq u\}.
\end{equation}

We can now state Theorem 1 of \cite{MP08}.

\begin{thm}[\cite{MP08} Theorem 1]
Let $G=(V,E)$ be a graph with countable vertices and bounded degree.
Suppose that $0<\gamma\leq\frac{1}{2}$ and $p(x,x)\geq\gamma$ for
all $x\in V$. If
\begin{equation}
n\geq
1+\frac{(1-\gamma)^2}{\gamma^2}\int_{4}^{4/\ep}{\frac{4du}{u\Phi^2(u)}},
\label{integral_condition}
\end{equation}
then
\begin{equation}
|p^n(x,y)|\leq\ep.
\end{equation}
\label{MP_theorem}
\end{thm}

Next we will prove the following claim:

\begin{clm}
Let $p_\om^n(x,y)$ be the probability that the random walk moves
from $x$ to $y$ in $n$ steps in the environment $\om$. Then there exist
positive constants $K_1,K_2$ depending only on $d$, and a natural
number $N$ such that for every $n>N$ and every $x,y\in\CP(\om)$
\begin{equation}
p_\om^{n}(x,y)\leq\frac{K_2}{(n-K_1)^{d/2}},~~~~~~P~a.s.
\end{equation}
\label{bound_of_transitions}
\end{clm}

\begin{proof}
We start by dealing with even steps of the Markov chain, and at the
end extend the argument to the odd ones. Since
$p^2(x,x)=\frac{1}{2d}$, we can use Theorem \ref{MP_theorem} with
$\gamma=\frac{1}{2d}$. Let $\om\in\Om_0$ and $S\subset\CP(\om)$
such that $|S|=n$. By Lemma \ref{lem_volume_surface} there exists a
positive constant $C$, such that at least one of the projections $\{\Pi^i(S)\}_{i=1}^d$ satisfy
$\Pi^{i}(S)\geq C\cdot n^{\frac{d-1}{d}}$. We will assume without
loss of generality that this holds for $i=1$. We now look at the set
\begin{equation}
\widetilde{S}=\big\{(x_1,x_2,\ldots,x_d):~(x_2,\ldots,x_d)\in\Pi^1(S),~~x_1=\max\{a:(a,x_2,x_3,\ldots,x_d)\in
S\}\big\}.
\end{equation}
We note that $|\widetilde{S}|=|\Pi^1(s)|\geq C n^{(d-1)/d}$.
In addition since $|\partial S|$ equals in our case to
$\frac{1}{2d}$ times the number of edges $e\in E$ with one end point in $S$ and the other in $S^c$, then $|\partial S|\geq \frac{1}{2d}|\widetilde{S}|$.
This is true since every element in $\widetilde{S}$ contributes at least one edge to the boundary. Using these two properties it follows that
there exists a positive constant $C_0$ such that
\begin{equation}
\Phi(u)\geq  C_0\frac{1}{u^{1/d}}, \label{Phi_estimation}
\end{equation}
and therefore
\begin{equation}
\begin{aligned}
1+(2d-1)^2\int_{4}^{4/\ep}{\frac{4du}{u\Phi^2(u)}} & \leq
1+(2d-1)^2\int_{4}^{4/\ep}{\frac{4u^{\frac{2}{d}-1}du}{C_0^2}}  \nonumber\\
&\leq  \left\lceil
1-\frac{2d(2d-1)^2}{c_0^2}4^{\frac{2}{d}}+\frac{2d(2d-1)^2}{c_0^2}4^{\frac{2}{d}}\ep^{-\frac{2}{d}}
\right\rceil . \nonumber \label{isoperimetric}
\end{aligned}
\end{equation}
Notice that $1-\frac{2d(2d-1)^2}{c_0^2}4^{\frac{2}{d}}$ is negative
for all but a finite number of dimensions, and therefore we can find a natural
number $\widetilde{K}_1(d)$ such that the last term in
\eqref{isoperimetric} is less than or equal to
\begin{equation}
n(\ep)\equiv\lceil
\widetilde{K}_1+\widetilde{K}_2\ep^{-\frac{2}{d}}\rceil,
\end{equation}
where
$\widetilde{K}_2=\widetilde{K}_2(d)=\frac{2d(2d-1)^2}{c_0^2}4^{\frac{2}{d}}$.
It therefore follows that
\begin{equation}
\ep\leq
\left(\frac{n(\ep)-\widetilde{K}_1-1}{\widetilde{K}_2}\right)^{-\frac{d}{2}}.
\end{equation}

Let $K_2=(\widetilde{K}_2)^{-\frac{d}{2}}$, since the condition in
Theorem \ref{MP_theorem} is fulfilled, for $P$ almost every
environment $\om$, for every $n>N$ and every $x,y\in\CP(\om)$
\begin{equation}
p_\om^{2n}(x,y)\leq
\frac{K_2}{(2n-\widetilde{K}_1-1)^{\frac{d}{2}}}.
\end{equation}

Moving to deal with transition probabilities for odd times, if
$n>N+1$ we have for $P$ almost every environment $\om$
\begin{equation}
\begin{aligned}
p^{2n+1}_\om(x,y)&=\sum_{z\in\CP(\om)}{p_\om(x,z)p^{2n}_\om(z,y)}\nonumber\\
&\leq \sum_{z\in\CP(\om)}{p_\om(x,z)\frac{K_2}{(2n-\widetilde{K}_1-1)^\frac{d}{2}}}\\
&= \frac{K_2}{(2n+1-\widetilde{K}_1-2)^\frac{d}{2}}.
\end{aligned}
\end{equation}
Taking $K_1=\widetilde{K}_1+2$ we get the desired inequality both
for even times and odd ones.
\end{proof}

We are now ready to prove Theorem \ref{tran_recu3}.

\begin{proof}[Proof of Theorem \ref{tran_recu3}]
Since our graph is connected, it is enough to show that
\begin{equation}
\sum_{n=0}^{\infty}{p^n(0,0)}<\infty.
\end{equation}
Using claim \ref{bound_of_transitions}, we get that for $P$ almost
every environment $\om\in\Om_0$
\begin{equation}
\sum_{n=0}^{\infty}{p_\om^n(0,0)}\leq\sum_{n=0}^{N-1}{p_\om^n(0,0)}+\sum_{n=N}^{\infty}{\frac{2K_2}{(2n-K_1)^{\frac{d}{2}}}}<\infty.
\end{equation}
\end{proof}


\section{Asymptotic behavior of the random walk}

In this section we prove asymptotic behavior of $\BE(\|X_n\|)$. This will be used in section 10 to prove the high
dimensional Central Limit Theorem. Therefore we assume here the
additional assumption, assumption \ref{assumption3}. The estimation
follows closely \cite{Ba03} with the following changes:
\begin{itemize}
\item The minor change is that we work in discrete time setting and not in continuous time.\\

\item The major change is that the average variance of the distance at the $n^{\mbox{th}}$ step of the
random walk is not bounded by 1 as in the percolation case.
Nevertheless we can show that if we assume in addition assumption
\ref{assumption3}, it is still bounded.
\end{itemize}

Other than that problem, in which we deal in part $(3)$ of Theorem \ref{asymptotic_X_n}, the rest of the proof
doesn't contain new ideas and follows \cite{Ba03}

\begin{thm}
Assuming assumption \ref{assumption3}, there exists a random
variable $c:\Om_0\rightarrow [0,\infty]$ which is finite almost
surely such that for $P$ almost every $\om\in\Om_0$
\begin{equation}
\BE_\om(\|X_n\|)\leq c\sqrt{n}~~~\forall n\in\BN.
\end{equation}
\label{asymptotic_X_n}
\end{thm}

We begin with a few definitions

\begin{defn}
Fix $\om\in\Om_0$. For $n\in\BN$ we denote $p^n(x,y)=P_\om(X_n=y|X_0=x)$ and
introduce the following functions, with the understanding that
$0\cdot\log(0)=0$:

\begin{enumerate}

\item $g_n:\CP(\om)\rightarrow\BR$, given by
\begin{equation}
g_n(x)=\frac{1}{2}\left(p^n(0,x)+p^{n-1}(0,x)\right).
\end{equation}

\item We define $M:\BN\rightarrow\BR^+$ by $M(0)=0$ and for
$n>0$ by:
\begin{equation}
M(n)=\frac{1}{2}\BE_\om(\|X_n\|+\|X_{n-1}\|):=\sum_{y\in\CP(\om)}{\|y\|g_n(y)}.
\end{equation}

\item We define $Q:\BN\rightarrow\BR^+$ by $Q(0)=0$ and for $n>0$
by:
\begin{equation}
Q(n)=-\sum_{y\in\CP(\om)}{g_n(y)\log(g_n(y))},
\end{equation}
i.e. $Q$ is the entropy of $g_n$.
\end{enumerate}
\end{defn}

In order to prove Theorem \ref{asymptotic_X_n}, we will prove some
inequalities introduced in the following proposition:

\begin{prop}
There exists $N=N(\om)\in\BN$ and constants $c_1,c_2,c_3,K_1<\infty$
such that for every $n>N$ we have

\begin{enumerate}

\item
\begin{equation}
Q(n)\geq c_1+\frac{d}{2}\log{(n-K_1)}, \label{distance_prop_1}
\end{equation}

\item
\begin{equation}
M(n)\geq c_2\cdot e^{\frac{Q(n)}{d}}, \label{distance_prop_2}
\end{equation}

\item
\begin{equation}
\sum_{x\in\CP(\om)}{\sum_{y\in\CP(\om)}{\ind_{\{y\in
N_x(\om)\}}{(g_n(x)+g_n(y))\|x-y\|^2}}}<\infty,
\label{distance_prop_3}
\end{equation}

\item
\begin{equation}
(M(n+1)-M(n))^2\leq c_3(Q(n+1)-Q(n)). \label{distance_prop_4}
\end{equation}

\end{enumerate}
\end{prop}
We note that we don't have any estimation on the tail of $N(\om)$.\\

\begin{proof}
$~$\\
\begin{enumerate}

\item From the definition of $Q(n)$ we have that
\begin{equation}
Q(n)\geq \inf_{y\in\CP(\om)}(-\log(g_n(y)))=-\sup_{y\in\CP(\om)}{(\log(g_n(y)))}.
\nonumber
\end{equation}
Using now Claim \ref{bound_of_transitions}, for sufficiently large $n$
we have $\forall y\in\CP(\om)$ that $g_n(y)\leq \frac{K_2}{(n-K_1)^{\frac{d}{2}}}$ and
therefore
\begin{equation}
Q(n)\geq -\log\left(
\frac{K_2}{(n-K_1)^\frac{d}{2}}\right)=-\log(K_2)+\frac{d}{2}\log(n-K_1).
\end{equation}
Taking $c_1=-\log(2K_2)$ we get the desired inequality.\\

\item Let $D_n=B_{2^n}(0)\backslash B_{2^{n-1}}(0)$ for $n>0$ and
$D(0)=\{0\}$, where $B_n(0)=\{x\in\BZ^d ~:~|x|\leq n\}$. Then for $0\leq a\leq 2$ we have:
\begin{equation}
\sum_{y\in\CP(\om)}{e^{-a\|y\|}} \leq \sum_{n=0}^{\infty}{\sum_{y\in
D_n}{e^{-a\cdot 2^n}}}\leq \sum_{n=0}^{\infty}{e^{-a\cdot 2^n}\cdot
c_{2.1}\cdot 2^{nd}}\leq c_{2.2}\cdot a^{-d},
\label{distance_inequality_1}
\end{equation}
where $c_{2.2}=c_{2.2}(d)$ depends on d. Indeed, the first
inequality is true since $a\leq 2$, the second inequality follows
from the fact that the set of points in $\CP(\om)$ with distance
greater than $2^{n-1}$ and less than $2^n$ is bounded by the number
of points in $\BZ^d$ with those properties, which is less than a
constant times $2^{nd}$. The proof of the last inequality follows by
separating the series into two parts, up to some $n_0$ and starting
from $n_0$, and then bounding the second one by a geometric series.
The proof of it can be found in the Appendix.

Since for every $u>0$ and every $\lambda\in\BR$ we have
$u(\log(u)+\lambda)\geq -e^{-1-\lambda}$, by taking
$\lambda=a\|y\|+b$ with $a\leq 2$ and $u=g_n(y)$ we get
\begin{equation}\label{eq:one_more}
\begin{aligned}
-Q(n)+aM(n)+b&=\sum_{y\in\CP(\om)}{g_n(y)\left(\log(g_n(y))+a\|y\|+b\right)}\\
&\geq-\sum_{y\in\CP(\om)}{e^{-1-a\|y\|-b}=-e^{-1-b}\sum_{y\in\CP(\om)}{e^{-a\|y\|}}}.
\end{aligned}
\end{equation}

Note that we actually used the last inequality only for those
$y\in\CP(\om)$ such that $g_n(y)>0$, and for $y\in\CP(\om)$ such
that $g_n(y)=0$ we used the fact that $0\geq -e^{-1-a\|y\|-b}$. Combining
\eqref{eq:one_more} and \eqref{distance_inequality_1} we get that
\begin{equation}
-Q(n)+aM(n)+b\geq -e^{-1-b} c_{2.2} a^{-d}.
\label{ditance_inequality_2.1}
\end{equation}
But for sufficiently large $n$ we have
\begin{equation}
\begin{aligned}
M(n)&=0\cdot g_n(0)+\sum_{y\in\CP(\om),y\neq 0}{d(0,y)g_n(y)}\\
\nonumber &\geq \sum_{y\in\CP(\om),y\neq 0}{g_n(y)}=1-g_n(0)\geq
\frac{1}{2}. \label{distance_inequality_2.2}
\end{aligned}
\end{equation}
Taking now $a=\frac{1}{M(n)}$ and $b=d\cdot \log\left(M(n)\right)$,
by \eqref{ditance_inequality_2.1} (and since by
\eqref{distance_inequality_2.2} we have $a\leq 2$) it follows that
\begin{equation}
-Q(n)+1+d\cdot\log(M(n))\geq -e^{-1} c_{2.2}=-c_{2.3}. \nonumber
\end{equation}
Note that $c_{2.3}=c_{2.3}(d)$ also depend on $d$. Rearranging the last inequality we
get that there exists a constant $c_2=c_2(d)$ such that
\begin{equation}
M(n)\geq c_2\cdot e^\frac{Q(n)}{d}. \nonumber
\end{equation}

\item We start by rearranging the sum as
\begin{equation}
\begin{aligned}
\sum_{x,y\in\CP(\om)}{\ind_{\{y\in
N_x(\om)\}}(g_n(x)+g_n(y))\|x-y\|^2}
&=2\sum_{x\in\CP(\om)}{g_n(x)\sum_{y\in N_x(\om)}{\|x-y\|^2}}\\
&= 2\sum_{e\in\{\pm
e_i\}_{i=1}^d}{\sum_{x\in\CP(\om)}{g_n(x)f_e^2(\gth^{x}\om)}}\\
&=2\sum_{e\in\{\pm e_i\}_{i=1}^d}{\left(E_\om(f_e^2\circ
\gth^{X_n})+E_\om(f_e^2\circ \gth^{X_{n-1}})\right)}. \nonumber
\end{aligned}
\end{equation}

In order to show that this sum is finite, we will use a theorem
taken from \cite{NS94}. Before we can  state the theorem we need the
following definitions:

Given a countable group $\Gamma$ we define
$l^1(\Gamma)=\{\mu=\sum_{\gamma\in\Gamma}\mu(\gamma)\gamma:\sum_{\gamma\in\Gamma}|\mu(\gamma)|<\infty\}$.
Let $(X,\mathcal{B},m)$ be a standard Lebesgue probability space,
and assume $\Gamma$ acts on $X$ by measurable automorphisms preserving
the probability measure $m$. This action induces a representation of
$\Gamma$ by isometries on the $L^p(X)$ spaces, $1\leq p\leq \infty$,
and this representation can be extended to $l^1(\Gamma)$ by $(\mu
f)(x)=\sum_{\gamma\in\Gamma}{\mu(\gamma)f(\gamma^{-1}x)}$. Let
$\mathcal{B}_1=\{A\in\mathcal{B}:m(\gamma A\bigtriangleup A)=0~\forall
\gamma\in\Gamma\}$ denote  the sub $\si$-algebra of invariant sets,
and denote by $E_1$ the conditional expectation with respect to
$\mathcal{B}_1$. We call a sequence $\nu_n\in l^1(\Gamma)$ a
pointwise ergodic sequence in $L^p$ if, for any action of $\Gamma$
on a Lebesgue space X which preserves a probability measure and for
every $f\in L^p(X)$, $\nu_nf(x)\rightarrow E_1f(x)$ for almost all
$x\in X$, and in the norm of $L^p(X)$. If $\Gamma$ is finitely
generated, let S be a finite generating symmetric set. S induces a
length function on $\Gamma$, given by
$|\gamma|=|\gamma|_S=\min\{n:\gamma=s_1s_2\ldots s_n,s_i\in S\}$,
and $|e|=0$. We can therefore define the following sequences:

\begin{defn}
$~$\\
\begin{enumerate}
\item[(i.)] $\tau_n=(\# S_n)^{-1}\sum_{w\in S_n} w$, where
$S_n=\{w:|w|=n\}$.

\item[(ii.)] $\tau '_n=\frac{1}{2}(\tau_n+\tau_{n+1})$.

\item[(iii.)] $\mu_n=\frac{1}{n+1}\sum_{k=0}^{n}{\tau_k}$.

\item[(iv.)] $\beta_n=(\# B_n)^{-1}\sum_{w\in B_n}w$, where $B_n=\{w:|w|\leq
n\}$.

\end{enumerate}
\end{defn}

We can now state the theorem:

\begin{thm} [Nevo, Stein 94]
Consider the free group $F_r$, $r\geq 2$. Then:
\begin{enumerate}
\item[1.] The sequence $\mu_n$ is a pointwise ergodic sequence in $L^p$,
for all $1\leq p< \infty$.

\item[2.] The sequence $\tau '_n$ is a pointwise ergodic sequence in
$L^p$, for $1<p<\infty$.

\item[3.] $\tau_{2n}$ converges to an operator of conditional
expectation with respect to an $F_r$-invariant sub $\si$-algebra.
$\beta_{2n}$ converges to the operator $E_1+((r-1)/r)E$, where E is
a projection disjoint from $E_1$. Given $f\in L^p(X)$, $1<p<\infty$,
the convergence is pointwise almost everywhere, and in the $L^p$
norm.
\end{enumerate}
\label{Birkhoff-generalization}
\end{thm}

We actually only need the second part of Theorem
\ref{Birkhoff-generalization}. Taking $S=\{\si_{\pm
e_i}\}_{i=1}^{d}$, we get that
\begin{equation}
2\sum_{e\in\{\pm e_i\}_{i=1}^d}{\left(E_\om(f_e^2\circ
\gth^{X_n})+E_\om(f_e^2\circ \gth^{X_{n-1}})\right)} \leq 4\sum_{e\in\{\pm e_i\}_{i=1}^d}{\tau_n'(f_e^2)}.\nonumber
\end{equation}
Using the additional assumption, we get that there exists
$1<p<\infty$ such that for every coordinate direction $e$, $f_e^2\in
L^p(\Om_0)$. Therefore by Theorem \ref{Birkhoff-generalization}
\begin{equation}
\lim_{n\rightarrow\infty}{4\sum_{e\in\{\pm
e_i\}_{i=1}^d}{\tau_n'(f_e^2)}}=E_1\left(4\sum_{e\in\{\pm
e_i\}_{i=1}^d}{f_e^2}\right), \nonumber
\end{equation}
exists. In addition, since $P$ is ergodic with respect to $\si_e$
for every coordinate direction $e$, there exists a constant $C$ such
that $4\sum_{e\in\{\pm e_i\}_{i=1}^d}{E_1(f_e^2)}=C$ $P$-almost
surely. Consequently, the original sequence converges to $C$ $P$-almost
surely, and therefore in particular it is $P$-almost surely bounded.

\item
\begin{equation}
M(n+1)-M(n)=\sum_{y\in\CP(\om)}{(g_{n+1}(y)-g_n(y))\|y\|}. \nonumber
\end{equation}
Using the discrete Gauss Green formula, this term equals to
\begin{equation}\label{Green_formula}
-\frac{1}{4d}\sum_{x,y\in\CP(\om)}{\ind_{\{y\in
N_x(\om)\}}(\|y\|-\|x\|)(g_n(y)-g_n(x))}.
\end{equation}
Indeed,rearranging the sums we get that
$\sum_{y\in\CP(\om)}{(g_{n+1}(y)-g_n(y))\|y\|}$ equals to
\begin{equation}
-\frac{1}{4d}\left[2d\sum_{y\in\CP(\om)}{\|y\|g_n(y)}+2d\sum_{x\in\CP(\om)}{\|x\|g_n(x)}\right.
\left.-2d\sum_{y\in\CP(\om)}{\|y\|g_{n+1}(y)}-2d\sum_{x\in\CP(\om)}{\|x\|g_{n+1}(x)}\right].
\nonumber
\end{equation}
Since all sums are finite and for every point in $x\in\CP(\om)$ we
have $|N_x(\om)|=2d<\infty$ we get that the last term is equal to
\begin{equation}
\begin{aligned}
-\frac{1}{4d}\left[\sum_{y\in\CP(\om)}{\|y\|g_n(y)\sum_{x\in\CP(\om)}{\ind_{y\in
N_x(\om)}}}+\sum_{x\in\CP(\om)}{\|x\|g_n(x)\sum_{y\in\CP(\om)}{\ind_{y\in
N_x(\om)}}} \right.\nonumber\\
\left.- \sum_{y\in\CP(\om)}{\|y\|\sum_{x\in\CP(\om)}{\ind_{y\in
N_x(\om)}g_n(x)}}-\sum_{x\in\CP(\om)}{\|x\|\sum_{y\in\CP(\om)}{\ind_{y\in
N_x(\om)}g_n(y)}}\right].\nonumber
\end{aligned}
\end{equation}
But again all sums are finite and therefore we can change the order
of summation getting the following presentation
\begin{equation}
\begin{aligned}
&-\frac{1}{4d}\sum_{x,y\in\CP(\om)}{\left[\ind_{y\in
N_x(\om)}\|y\|g_n(y)-\ind_{y\in N_x(\om)}\|x\|g_n(y)-\ind_{y\in
N_x(\om)}\|y\|g_n(x)+\ind_{y\in N_x(\om)}\|x\|g_n(x)\right]}
\nonumber\\
=&-\frac{1}{4d}\sum_{x,y\in\CP(\om)}{\ind_{\{y\in
N_x(\om)\}}(\|y\|-\|x\|)(g_n(y)-g_n(x))}.
\end{aligned}
\end{equation}

Using \eqref{Green_formula} and the triangle inequality we get that
$M(n+1)-M(n)$ is less or equal than
\begin{equation}
\frac{1}{4d}\sum_{x,y\in\CP(\om)}{\ind_{\{y\in
N_x(\om)\}}\|x-y\|\left|g_n(y)-g_n(x)\right|}.\nonumber
\end{equation}
Therefore by the Cauchy Schwartz inequality
\begin{equation}
\begin{aligned}
M(n+1)-M(n)\leq
\frac{1}{4d}&\left(\sum_{x,y\in\CP(\om)}{\ind_{\{y\in
N_x(\om)\}}(g_n(x)+g_n(y))\|x-y\|^2}\right)^\frac{1}{2}\nonumber\\
\cdot &\left(\sum_{x,y\in\CP(\om)}{\ind_{\{y\in
N_x(\om)\}}\frac{(g_n(y)-g_n(x))^2}{g_n(y)+g_n(x)}}\right)^\frac{1}{2}.\nonumber
\label{distance_inequality_3.1}
\end{aligned}
\end{equation}

The first sum here is exactly the same sum from (\ref{distance_prop_3}) and therefore
is finite, so there exists a positive constant $c_{3.1}=c_{3.1}(d)$ such that
$M(n+1)-M(n)$ is less or equal to
\begin{equation}
c_{3.1}\left(\sum_{x,y\in\CP(\om)}{\ind_{\{y\in
N_x(\om)\}}\frac{(g_n(y)-g_n(x))^2}{g_n(y)+g_n(x)}}\right)^\frac{1}{2}.
\nonumber
\end{equation}
Using the fact that for every $u,v>0$
\begin{equation}
\frac{(u-v)^2}{u+v}\leq(u-v)\left(\log(u)-\log(v)\right). \nonumber
\end{equation}
We get that $M(n+1)-M(n)$ is less or equal than
\begin{equation}
c_{3.1}\left(\sum_{x,y\in\CP(\om)}{\ind_{\{y\in
N_x(\om)\}}\Big(g_n(y)-g_n(x)\Big)\Big(\log(g_n(y))-\log(g_n(x))\Big)}\right)^\frac{1}{2}.
\nonumber
\end{equation}
Using the discrete Gauss Green formula in the other direction, the
last term equals to
\begin{equation}
\sqrt{4d}c_{3.1}\left(-\sum_{y\in\CP(\om)}{\Big(\log(g_n(y))+1\Big)\Big(g_{n+1}(y)-g_n(y)\Big)}\right)^\frac{1}{2}.
\nonumber
\end{equation}
Since $1-x+\log(x)\leq 0$ for all $x>0$ we get that the last term is
less or euqal to
\begin{equation}
\sqrt{4d}c_{3.1}\left(-\sum_{y\in\CP(\om)}{\Big(g_{n+1}(y)-g_n(y)\Big)\log(g_n(y))+g_{n+1}(y)\log\left(\frac{g_{n+1}(y)}{g_n(y)}\right)}\right)^\frac{1}{2}.\nonumber
\end{equation}
But this is exactly
\begin{equation}
\sqrt{4d}c_{3.1}\Big(Q(n+1)-Q(n)\Big)^\frac{1}{2}. \nonumber
\end{equation}
By taking $c_3=(\sqrt{4d}c_{3.1})^2$ gives the desired inequality.
\end{enumerate}
\end{proof}

\begin{proof}[Proof of Theorem \ref{asymptotic_X_n}]
Let $R(n):\BN\rightarrow\BR$ be defined by
\begin{equation}
R(n)=\frac{1}{d}\Big(Q(n)-c_1-\frac{d}{2}\log(n-K_1)\Big),
\end{equation}
for $n>\lceil K_1\rceil+1$ and $R(n)=0$ for $n\leq\lceil
K_1\rceil+1$. By \eqref{distance_prop_2} for sufficiently large $n$ we
have
\begin{equation}
M(n)\geq c_2\cdot e^\frac{Q(n)}{d}=c_2\cdot
e^{R(n)+\frac{c_1}{d}+\frac{1}{2}\log(n-K_1)}=c_{4.1}e^{R(n)}\sqrt{n-K_1}.
\label{distance_inequality_last1}
\end{equation}

On the other hand, let $N\in\BN$ be such that for all $n>N$
inequalities (\ref{distance_prop_1}-\ref{distance_prop_4}) hold,
then for every $n>N$ we have (set $c_{4.3}=\sqrt{c_3}$)
\begin{equation}
\begin{aligned}
M(n) & = \sum_{k=1}^{N}{M(k)-M(k-1)}+\sum_{k=N+1}^{n}{M(k)-M(k-1)}
\nonumber \\
& \leq  c_{4.2}+c_{4.3}\cdot
\sum_{k=N+1}^{n}{\Big(Q(k)-Q(k-1)\Big)^\frac{1}{2}} \nonumber \\
& =
c_{4.2}+c_{4.3}\sqrt{d}\sum_{k=N+1}^{n}{\left(R(k)-R(k-1)+\frac{1}{2}\log\left(\frac{k-K_1}{k-1-K_1}\right)\right)^{\frac{1}{2}}}.
\nonumber
\end{aligned}
\end{equation}
Using the inequality $(a+b)^\frac{1}{2}\leq
b^\frac{1}{2}+\frac{a}{(2b)^\frac{1}{2}}$, we find that this is less
than or equal to
\begin{equation}
c_{4.2}+c_{4.3}\sum_{k=N+1}^{n}{\left[\frac{1}{\sqrt{2}}\log^\frac{1}{2}\left(\frac{k-K_1}{k-1-K_1}\right)+
\frac{R(k)-R(k-1)}{\log^\frac{1}{2}\left(\frac{k-K_1}{k-1-K_1}\right)}\right]},
\nonumber
\end{equation}
which can be written (using discrete integration by parts) as
\begin{equation}
\begin{aligned}
&c_{4.2}+c_{4.3}\sum_{k=N+1}{n}{\frac{1}{\sqrt{2}}\log^\frac{1}{2}\left(\frac{k-K_1}{k-1-K_1}\right)}+
c_{4.3}\sum_{k=N+1}^{n}{\left[\frac{R(k)}{\log^\frac{1}{2}\left(\frac{k+1-K_1}{k-K_1}\right)}-\frac{R(k-1)}{\log^\frac{1}{2}\left(\frac{k-K_1}{k-1-K_1}\right)}\right]}
\nonumber\\
&-c_{4.3}\sum_{n=N+1}^{n}{R(k)\left[\frac{1}{\log^\frac{1}{2}\left(\frac{k+1-K_1}{k-K_1}\right)}-\frac{1}{\log^\frac{1}{2}\left(\frac{k-K_1}{k-1-K_1}\right)}\right]}.
\nonumber
\end{aligned}
\end{equation}
Since \eqref{distance_prop_1} holds $R(k)$ is non negative and
therefore the last sum is positive. Consequently we get
\begin{equation}
\label{distance_long_calc_1}
\begin{aligned}
M(n)\leq
c_{4.2}&+c_{4.3}\sum_{k=N+1}^{n}{\frac{1}{\sqrt{2}}\log^\frac{1}{2}\left(\frac{k-K_1}{k-1-K_1}\right)}\\
&+
c_{4.3}\sum_{k=N+1}^{n}{\left[\frac{R(k)}{\log^\frac{1}{2}\left(\frac{k+1-K_1}{k-K_1}\right)}-\frac{R(k-1)}{\log^\frac{1}{2}\left(\frac{k-K_1}{k-1-K_1}\right)}\right]}.
\end{aligned}
\end{equation}
Using the fact that
\begin{equation}
\log\left(\frac{k-K_1}{k-1-K_1}\right)=\log\left(1+\frac{1}{k-1-K_1}\right)<\frac{1}{k-1-K_1}.
\nonumber
\end{equation}
The first sum in \eqref{distance_long_calc_1} is less than
\begin{equation}
\sum_{k=N+1}^{n}{\frac{1}{(k-1-K_1)^\frac{1}{2}}}\leq
c_{4.4}\sqrt{n-K_1}. \nonumber
\end{equation}
Therefore we find that
\begin{equation}
\begin{aligned}
M(n) & \leq  c_{4.2}+c_{4.3} c_{4.4}\sqrt{n-K_1}+c_{4.3}
\sum_{k=N+1}^{n}{\left[\frac{R(k)}{\log^\frac{1}{2}\left(\frac{k+1-K_1}{k-K_1}\right)}-\frac{R(k-1)}{\log^\frac{1}{2}\left(\frac{k-K_1}{k-1-K_1}\right)}\right]}
\nonumber\\
& =  c_{4.2}+c_{4.3}
c_{4.4}\sqrt{n-K_1}+c_{4.3}\left[\frac{R(n)}{\log^\frac{1}{2}\left(\frac{n+1-K_1}{n-K_1}\right)}-\frac{R(N+1)}{\log^\frac{1}{2}\left(\frac{N+1-K_1}{N-K_1}\right)}\right]
\nonumber\\
& \leq  c_{4.2}+c_{4.3} c_{4.4}\sqrt{n-K_1}+c_{4.3}
\frac{R(n)}{\log^\frac{1}{2}\left(\frac{n+1-K_1}{n-K_1}\right)}
\nonumber\\
& \leq  c_{4.2}+c_{4.3} c_{4.4}\sqrt{n-K_1}+c_{4.3}\cdot
c_{4.5} R(n)\sqrt{n-K_1}. \nonumber
\end{aligned}
\end{equation}
We can thus find a constant $c_{4.6}$ such that for all sufficiently
large $n$
\begin{equation}
M(n)\leq c_{4.6}[1+R(n)]\sqrt{n-K_1}.
\label{distance_inequality_last2}
\end{equation}
So by \eqref{distance_inequality_last1} and
\eqref{distance_inequality_last2} we have that for sufficiently
large $n$
\begin{equation}
c_{4.1}e^{R(n)}\sqrt{n-K_1}\leq M(n)\leq c_{4.6}[1+R(n)]\sqrt{n-K_1}.
\nonumber
\end{equation}
It follows that $R(n)$ must be a bounded function, and therefore we can
find constants $c_{4.7},c_{4.8}$ such that for sufficiently large n
\begin{equation}
c_{4.7}\sqrt{n-K_1} \leq M(n) \leq c_{4.8}\sqrt{n-K_1}. \nonumber
\end{equation}
Consequently, since
\begin{equation}
M(n)=\frac{1}{2}\Big[E_\om(\|X_n\|)+E_\om(\|X_{n-1}\|) \nonumber\Big],
\end{equation}
it follows that there exists a constant $c>0$ such that for $P$ almost every $\om\in\Om_0$
\begin{equation}
E_\om(\|X_n\|)\leq c\sqrt{n}~~~~\forall n\in\BN. \nonumber
\end{equation}
\end{proof}


\section{Corrector - Construction and harmonicity}

In this section, we adapt the construction presented in \cite{BB06}
(which in turn adapts the construction of Kipnis and Varadhan
\cite{KV86}) into our analysis.

We start with the following observation concerning the Markov chain
"on environments".

\begin{lem}
For every bounded measurable function $f:\Om_0\rightarrow \BR$ and
every $x\in N_0(\om)$ we have
\begin{equation}
\BE_P\left[(f\circ \gth_x)\ind_{\{x\in
N_0(\om)\}}\right]=\BE_P[f\ind_{\{-x\in N_0(\om)\}}].
\label{lem_chain_environments}
\end{equation}
As a consequence, $P$ is reversible and, in particular, stationary
for the Markov kernel $\Lambda$ defined in (\ref{Lambda}).
\end{lem}

\begin{proof}
We will first prove (\ref{lem_chain_environments}). Up to the factor
$\BP(\Om_0)$, we need to show that
\begin{equation}
\BE_Q[f\circ\gth_x\ind_{\Om_0}\ind_{\{x\in
N_0(\om)\}}]=\BE_Q[f\ind_{\Om_0}\ind_{\{-x\in N_0(\om)\}}].
\label{lem_chain_environments2}
\end{equation}
This will follow from the fact that $\ind_{\{x\in
N_0(\om)\}}\ind_{\Om_0}=\left(\ind_{\{-x\in
N_0(\om)\}}\ind_{\Om_0}\right)\circ\gth_x$. This observation implies
that
\begin{equation}
f\circ\gth_x\ind_{\Om_0}\ind_{\{x\in
N_0(\om)\}}=\left(f\ind_{\Om_0}\ind_{\{-x\in
N_0(\om)\}}\right)\circ\gth_x, \label{lem_chain_environments3}
\end{equation}
and (\ref{lem_chain_environments2}) follows from
(\ref{lem_chain_environments3}) by the shift invariance of $Q$. From
(\ref{lem_chain_environments}) we deduce that for any bounded
measurable functions $f,g:\Om\rightarrow\BR$,

\begin{equation}
\BE_P[f\cdot (\Lambda g)]=\BE_P[g\cdot (\Lambda f)],
\label{inner_product_property}
\end{equation}
where $\Lambda f:\Om_0\rightarrow\BR$ is the function
\begin{equation}
(\Lambda f)(\om)=\frac{1}{2d}\sum_{x\in\BZ^d}{\left(\ind_{\{x\in
N_0(\om)\}}f(\gth_x\om)\right)}. \label{Lambda_function}
\end{equation}
Indeed
\[\BE_P[f\cdot (\Lambda g)]=\frac{1}{2d}\sum_{x\in\BZ^d}\BE_P[f\cdot g\circ\gth_x\ind_{\{x\in N_0(\om)\}}].\]
Applying (\ref{lem_chain_environments}) we get
\[\BE_P[f\cdot (\Lambda g)]=\frac{1}{2d}\sum_{x\in\BZ^d}\BE_P[f\circ \gth_{-x}\ind_{\{-x\in N_0(\om)\}}\cdot g]=\BE_P[(\Lambda f)\cdot g], \]
where we replaced the sign in the sum in order to cancel the
negative sign inside the sum. But (\ref{inner_product_property}) is
the definition of reversibility. Setting $f=1$ and noting that
$\Lambda f=1$, we get that for every bounded measurable function
$g:\Om\rightarrow\BR$
\[\BE_P[\Lambda g]=\BE_P[g],\]
and therefore $P$ is stationary with respect to the Markov kernel
$\Lambda$.
\end{proof}

\subsection{The Kipnis-Varadhan Construction}
$~$\\

Next we will adapt the construction of Kipnis and Varadhan
\cite{KV86} cited from \cite{BB06} to the present analysis. Let
$L^2=L^2(\Om_0,\FB,P)$ be the space of all Borel-measurable square
integrable functions on $\Om_0$. We will use the notation $L^2$ both
for $\BR$-valued functions as well as for $\BR^d$-valued functions.
We equip $L^2$ with the inner product $(f,g)=\BE_P[fg]$, when for
vector valued functions on $\Om$ we interpret $"fg"$ as the scalar
product of $f$ and $g$. Let $\Lambda$ be the operator defined by
(\ref{Lambda_function}), and we expand the definition to vector
valued functions by letting $\Lambda$ act like a scalar, i.e.,
independently for each component. From
(\ref{inner_product_property}) we get that
\begin{equation}
(f,\Lambda g)=(\Lambda f,g),
\end{equation}
and so $\Lambda$ is symmetric. In addition, for every $f\in L^2$ we
have
\[|(f,\Lambda f)|\leq\frac{1}{2d}\sum_{x\in\BZ^d}{|(f,\ind_{\{x\in N_0(\om)\}}f\circ\gth_x)|}=
\frac{1}{2d}\sum_{x\in\BZ^d}{|(f\ind_{\{x\in
N_0(\om)\}},\ind_{\{x\in N_0(\om)\}}f\circ\gth_x)|}.\]
Using the Cauchy-Schwartz inequality this is less than or equal to
\[\frac{1}{2d}\sum_{x\in\BZ^d}{(f\ind_{\{x\in N_0(\om)\}},f\ind_{\{x\in N_0(\om)\}})^{1/2}\cdot(\ind_{\{x\in N_0(\om)\}}f\circ\gth_x,\ind_{\{x\in N_0(\om)\}}f\circ\gth_x)^{1/2}},\]
which equals
\[\frac{1}{2d}\sum_{x\in\BZ^d}{(f,f\ind_{\{x\in N_0(\om)\}})^{1/2}\cdot(1,\ind_{\{x\in N_0(\om)\}}f^2\circ\gth_x)^{1/2}}.\]
Using (\ref{lem_chain_environments}) we find that this this equals
\[\frac{1}{2d}\sum_{x\in\BZ^d}{(f,f\ind_{\{x\in N_0(\om)\}})^{1/2}\cdot(f,f\ind_{\{-x\in N_0(\om)\}})^{1/2}}\leq\frac{1}{2d}\sum_{x\in\BZ^d}{(f,f\ind_{\{x\in N_0(\om)\}})}=(f,f),\]
and so $~\|\Lambda\|_{L^2}\leq 1$. In particular, $\Lambda$ is self
adjoint and $sp(\Lambda)\subseteq [-1,1]$.

Let $V:\Om_0\rightarrow\BR^d$ be the local drift at the origin i.e,
\begin{equation}
V(\om)=\frac{1}{2d}\sum_{x\in\BZ^d}{x\ind_{\{x\in N_0(\om)\}}}.
\label{drift_equation}
\end{equation}

If the second moment of $f_e$ exists for every $e\in\CE$, then $V\in L^2$. Indeed
\[(V,V)=\sum_{e\in\CE}{(V\cdot e,V\cdot e)},\]
and
\[(V\cdot e,V\cdot e)=\frac{1}{2d}\BE_P[(V\cdot e)^2]=\frac{1}{2d}\BE_P[(f_e)^2+(f_{-e})^2],\]
which is finite if the second moments exist. For each $\ep>0$, let
$\psi_\ep:\Om_0\rightarrow\BR^d$ be the solution of
\begin{equation}
(1+\ep-\Lambda)\psi_\ep=V. \label{defining_psi_ep}
\end{equation}
This is well defined since $sp(\Lambda)\subset[-1,1]$, so for every
$\ep>0$ we get $sp(1+\ep+\Lambda)\subset[\ep,2+\ep]$. In addition we
get that $\psi_\ep\in L^2$ for all $\ep>0$. The following theorem is
the main result concerning the corrector:

\begin{thm}
There is a function $\chi:\BZ^d\times\Om_0\rightarrow\BR^d$ such
that for every $x\in\BZ^d$,
\begin{equation}
\lim_{\ep\searrow
0}{\ind_{\{x\in\CP(\om)\}}(\psi_\ep\circ\gth_x-\psi_\ep)}=\chi(x,\cdot),~~~~~in~
L^2.
\end{equation}
Moreover, the following properties hold:
\begin{enumerate}

\item (Shift invariance) For $P$-almost every $\om\in\Om_0$
\begin{equation}
\chi(x,\om)-\chi(y,\om)=\chi(x-y,\gth_y(\om)),
\label{shift_invariance}
\end{equation}
for all $x,y\in\CP(\om)$.

\item (Harmonicity) For $P$-almost every $\om\in\Om_0$, the function
\begin{equation}
x\mapsto\chi(x,\om)+x,
\end{equation}
is harmonic with respect to the transition probability given in
(\ref{transition_probability})

\item (Square integrability) There exists a constant $C<\infty$ such that
\begin{equation}
\|[\chi(x+y,\cdot)-\chi(x,\cdot)]\ind_{\{x\in\CP(\om)\}}(\ind_{\{y\in
N_0(\om)\}}\circ\gth_x)\|_2<C, \label{square_integrability}
\end{equation}
for all $x,y\in\BZ^d$.
\end{enumerate}
\label{corrector_thm}
\end{thm}

The rest of this section deals with proving Theorem
\ref{corrector_thm}. The proof is based on spectral calculus and
closely follows the corresponding arguments from \cite{BB06} and
\cite{KV86}.


%

\subsection{Spectral calculation}
$~$\\

Let $\mu_{\Lambda,V}=\mu_V$ denote the spectral measure of $\Lambda:
L^2\rightarrow L^2$ associated with the function $V$. i.e, for every
bounded, continuous function $\Phi:[-1,1]\rightarrow\BR$, we have
\begin{equation}
(V,\Phi(\Lambda)V)=\int_{-1}^{1}{\Phi(\lambda)\mu_V(d\lambda)}.
\end{equation}
Since $\Lambda$ acts as a scalar, $\mu_V$ is the sum of the "usual"
spectral measures for the Cartesian components of $V$. In the
integral, we used the fact that $sp(\Lambda)\subset[-1,1]$, and
therefore the measure $\mu_V$ is supported entirely on $[-1,1]$. The
first observation, made already by Kipnis and Varadhan, is stated as
follows:

\begin{lem}
Assume that the second moments of $\{f_{\pm e_i}\}_{i=1}^{d}$ are
finite, then
\begin{equation}
\int_{-1}^{1}{\frac{1}{1-\lambda}\mu_V(d\lambda)}<\infty.
\label{spectral_measure}
\end{equation}
\end{lem}

\begin{proof}
The proof follows the proof of Lemma 2.3 in \cite{BB06}. Let $f\in
L^2$ be a bounded real-valued function. Using
\eqref{lem_chain_environments} we get
\begin{equation}
\sum_{x\in\BZ^d}{x\BE_P[f\ind_{\{x\in
N_0(\om)\}}]}=\frac{1}{2}\sum_{x\in\BZ^d}x\BE_P[(f-f\circ\gth_x)\ind_{\{x\in
N_0(\om)\}}]. \label{symmetry_for_lem}
\end{equation}
Hence, for every $a\in\BZ^d$ we get
\begin{equation}
\begin{aligned}
(f,a\cdot V)& =  \frac{1}{2d}\sum_{x\in\BZ^d}{x\cdot a}\BE_P[f\ind_{\{x\in N_0(\om)\}}]\nonumber\\
& =  \frac{1}{2}\frac{1}{2d}\sum_{x\in\BZ^d}{x\cdot a}\BE_P[(f-f\circ\gth_x)\ind_{\{x\in N_0(\om)\}}]\nonumber\\
& \leq  \frac{1}{2}\left(\frac{1}{2d}\sum_{x\in\BZ^d}{(x\cdot a)^2 P(x\in N_0(\om))}\right)^{1/2} \nonumber\\
& \cdot
\left(\frac{1}{2d}\sum_{x\in\BZ^d}{\BE_P[(f-f\circ\gth_x)^2\ind_{\{x\in
N_0(\om)\}}]}\right)^{1/2}, \nonumber
\end{aligned}
\end{equation}
where we used \eqref{symmetry_for_lem} in the second equality, and
the Cauchy-Schwartz inequality for the inequality. Using the
assumption that the second moments of $f_e$ exist for every
$e\in\CE$, the first term on the right hand side is less than a
finite constant times $|a|$. On the other hand, the second term,
using \eqref{lem_chain_environments}, can be written as follows:
\begin{equation}
\begin{aligned}
\frac{1}{2d}\sum_{x\in\BZ^d}{\BE_P((f-f\circ\gth_x)^2\ind_{\{x\in
N_0(\om)\}})}&\\
&=2\frac{1}{2d}\sum_{x\in\BZ^d}{\BE_P(f(f-f\circ\gth_x)\ind_{\{x\in
N_0(\om)\}})} \nonumber\\
&= 2(f,(1-\Lambda)f).\nonumber
\end{aligned}
\end{equation}
From the assumption that the second moments exist, there exists a
constant $C_0<\infty$ such that for all bounded $f\in L^2$,
\begin{equation}
|(f,a\cdot V)|\leq C_0|a|\left(f,(1-\Lambda)f\right)^{1/2}.
\label{bound_for_(f,aV)}
\end{equation}
Applying \eqref{bound_for_(f,aV)} for $f$ of the form
$f=a\cdot\psi(\Lambda)V$,  where $a\in\BR^d$, and
$\Psi:[-1,1]\rightarrow\BR$ is a bounded continuous function,
summing over coordinate vectors in $\BR^d$ and invoking
\eqref{spectral_measure}, we get that
\begin{equation}
\begin{aligned}
\left|\int_{-1}^{1}{\psi(\lambda)\mu_V(d\lambda)}\right|
&=  \left|\sum_{i=1}^{d}{(V\cdot e_i,\psi(\Lambda)V\cdot e_i)}\right| \nonumber\\
&\leq  \sum_{i=1}^{d}{\big|(V\cdot e_i,\psi(\Lambda)V\cdot e_i)\big|} \nonumber\\
&\leq  C_0\sum_{i=1}^{d}{\big(V\cdot e_i,\psi(\Lambda)^2(1-\Lambda)V\cdot e_i\big)^{1/2}} \nonumber\\
&\leq  C_0\sqrt{d}\left(\sum_{i=1}^{d}{\big(V\cdot e_i,\psi(\Lambda)^2(1-\Lambda)V\cdot e_i\big)}\right)^{1/2}  \nonumber\\
&=
C_0\sqrt{d}\left(\int_{-1}^{1}{\psi(\lambda)^2(1-\lambda)\mu_V(d\lambda)}\right)^{1/2}.
\nonumber
\end{aligned}
\end{equation}
Substituting
$\psi_\ep(\lambda)=\min{\{\frac{1}{\ep},\frac{1}{1-\lambda}\}}$ for
$\psi$ and noting that $(1-\lambda)\psi_\ep(\lambda)\leq 1$, we get
\begin{equation}
\int_{-1}^{1}{\psi_\ep(\lambda)\mu_V(d\lambda)}\leq
C_0\sqrt{d}\left(\int_{-1}^{1}{\psi_\ep(\lambda)\mu_V(d\lambda)}\right)^{1/2},
\end{equation}
and therefore
\begin{equation}
\int_{-1}^{1}{\psi_\ep(\lambda)\mu_V(d\lambda)}\leq d\cdot C_0^2.
\end{equation}
Now, the Monotone Convergence Theorem implies that
\begin{equation}
\int_{-1}^{1}{\frac{1}{1-\lambda}\mu_V(d\lambda)}=\lim_{\ep\searrow
0}{\int_{-1}^{1}{\psi_\ep(\lambda)\mu_V(d\lambda)}}=\sup_{\ep>0}{\int_{-1}^{1}{\psi_\ep(\lambda)\mu_V(d\lambda)}}\leq
d\cdot C_0^2<\infty,
\end{equation}
proving the desired claim.
\end{proof}

We now turn to prove the following lemma, also taken from
\cite{BB06}:

\begin{lem}
Let $\psi_\ep$ be defined as in \eqref{defining_psi_ep}, i.e, the
solution of $(1+\ep-\Lambda)\psi_\ep=V$. Then
\begin{equation}
\lim_{\ep\searrow 0}{\ep\|\psi_\ep\|_2^2}=0.
\label{psi_squared_limit}
\end{equation}
In addition, for every $x\in\BZ^d$ let
\begin{equation}
G_x^{(\ep)}(\om)=\ind_{\Om_0}(\om)\cdot\ind_{\{x\in
N_0(\om)\}}(\om)\cdot (\psi_\ep\circ\gth_x(\om)-\psi_\ep(\om)).
\label{G_definition}
\end{equation}
Then for all $x,y\in \BZ^d$,
\begin{equation}
\lim_{\ep_1,\ep_2\searrow
0}{\|G_x^{(\ep_1)}\circ\gth_y-G_x^{(\ep_2)}\circ\gth_y\|_2}=0.
\label{lim_of_G_exist}
\end{equation}
\end{lem}

\begin{proof}
The proof follows the proof in \cite{BB06}. From the definition of
$\psi_\ep$ we have,
\begin{equation}
\ep\|\psi_\ep\|_2^2=\int_{-1}^{1}{\frac{\ep}{(1+\ep-\lambda)^2}\mu_V(d\lambda)}.
\end{equation}
The integrand is dominated by $\frac{1}{1-\lambda}$ and in addition
tends to zero as $\ep\searrow 0$ in the support of $\mu_V$. Then
\eqref{psi_squared_limit} follows by the Dominated Convergence
Theorem. The second part of the claim is proved similarly: First we
get rid of the $y$-dependence by noting the following. Due to the
fact that $G_x^{\ep}\circ\gth^y\neq 0$ ensure that $y\in\CP(\om)$,
and since $P$ is invariant under translation of the form
$\gth^z~~z\in\CP(\om)$ we get that:
\begin{equation}
\|G_x^{(\ep_1)}\circ\gth_y-G_x^{(\ep_2)}\circ\gth_y\|_2=\|G_x^{(\ep_1)}-G_x^{(\ep_2)}\|_2.
\label{norm_equality_G}
\end{equation}
Therefore, averaging the square of \eqref{norm_equality_G} over
$x\in N_0(\om)$ we find that
\begin{equation}
\begin{aligned}
\frac{1}{2d}\sum_{x\in
N_0(\om)}{\|G_x^{(\ep_1)}\circ\gth_y-G_x^{(\ep_2)}\circ\gth_y\|_2^2}
& = \frac{1}{2d}\sum_{x\in N_0(\om)}{\|G_x^{(\ep_1)}-G_x^{(\ep_2)}\|_2^2}\nonumber\\
&=\frac{1}{2d}\sum_{x\in
N_0(\om)}{\BE_P\Big[(G_x^{(\ep_1)}-G_x^{(\ep_2)})^2\Big]}  \\
&= \frac{1}{2d}\sum_{x\in \BZ^d}{\BE_P\Big[\ind_{\Om_0}\ind_{\{x\in
N_0{(\om)}\}}(\Psi\circ\gth_x-\Psi)^2\Big]},
\end{aligned}
\end{equation}
where $\Psi=\psi_{\ep_1}-\psi_{\ep_2}$. Expanding the last
expression we see that it equals to:
\begin{equation}
\frac{1}{2d}\sum_{x\in \BZ^d}{\BE_P[\ind_{\Om_0}\ind_{\{x\in
N_0{(\om)}\}}(\Psi^2\circ\gth_x+\Psi^2-2\Psi\cdot\Psi\circ\gth_x)]}.
\end{equation}
Since $P$ is stationary under translation $\gth_x$ when $x\in
N_0(\om)$, we get that it can be written as
\begin{equation}
2(\Psi,\Psi)-2\left(\Psi,\frac{1}{2d}\sum_{x\in\BZ^d}{\BE_P(\ind_{\Om_0}\ind_{\{x\in
N_0{(\om)}\}}\Psi\circ\gth_x)}\right)=2(\Psi,(1-\Lambda)\Psi).
\end{equation}
Finally we evaluate $(\Psi,(1-\Lambda)\Psi)$:
\begin{equation}
\begin{aligned}
\left(\psi_{\ep_1}-\psi_{\ep_2},(1-\Lambda)(\psi_{\ep_1}-\psi_{\ep_2})\right)&
=\int_{-1}^{1}{\left(\frac{1}{(1+\ep_1-\lambda)^2}-\frac{1}{(1+\ep_2-\lambda)^2}\right)(1-\lambda)\mu_v(d\lambda)}\nonumber\\
&=\int_{-1}^{1}{\frac{(\ep_1-\ep_2)^2(1-\lambda)}{(1+\ep_1-\lambda)^2(1+\ep_2-\lambda)^2}\mu_V(d\lambda)}.\nonumber
\end{aligned}
\end{equation}
The integrand here is again bounded by $\frac{1}{1-\lambda}$ for all
$\ep_1,\ep_2>0$, and it tends to zero as $\ep_1,\ep_2\searrow 0$.
The claim now follows by the Dominated Convergence Theorem.\\
\end{proof}

Now we are finally ready to prove Theorem \ref{corrector_thm}.\\

\begin{proof}[Proof of Theorem \ref{corrector_thm}]
Again we closely follow the proof of Theorem 2.2 in \cite{BB06}. Let
$G^{\ep}_x\circ\gth_y$ be as in \eqref{G_definition}. Using
\eqref{lim_of_G_exist} we know that $G^{\ep}_x\circ\gth_y$ converges
in $L^2$ as $\ep\searrow 0$. We denote the limit by
$G_{y,y+x}=\lim_{\ep\searrow 0}{G_x^\ep\circ\gth_y}$. Since
$G^{\ep}_x\circ\gth_y$ is a gradient field on $\CP(\om)$, we have
$G_{y,y+x}(\om)+G_{y+x,y}(\om)=0$ and, more generally
$\sum_{k=0}^{n-1}{G_{x_k,x_{k+1}}}=0$ whenever $(x_0,x_1,\ldots,x_n)$
is a closed loop on $\CP(\om)$. Thus we may define
\begin{equation}
\chi(x,\om):=\sum_{k=0}^{n-1}{G_{x_k,x_{k+1}}(\om)},
\label{chi_definition}
\end{equation}
where $(x_0,x_1,\ldots,x_n)$ is a "nearest neighbor" (in the sense
of $x_i\in N_{x_{i-1}}(\om)$) path on $\CP(\om)$ connecting $x_0=0$
to $x_n=x$. By the above "loop" conditions, the definition is
independent of this path for almost every $\om\in\Om_0\cap\{\om:x\in\CP(\om)\}$.
The shift invariance \eqref{shift_invariance} will now follow from
the definition of $\chi$ and the fact that $G_{x,x+y}=G_{0,y}\circ\gth_x$.
In light of the shift invariance, to prove the harmonicity of
$x\mapsto x+\chi(x,\om)$ it is sufficient to show that, almost
surely,
\begin{equation}
\frac{1}{2d}\sum_{x\in
N_0(\om)}{\left[x+\chi(x,\cdot)\right]=\chi(0,\cdot)},
\end{equation}
which can be written as:
\begin{equation}
\frac{1}{2d}\sum_{x\in
N_0(\om)}{\left[\chi(0,\cdot)-\chi(x,\cdot)\right]}=V(\om).
\end{equation}
By the definition of $\chi$ we have for $x\in N_0(\om)$ that
$\chi(x,\cdot)-\chi(0,\cdot)=G_{0,x}$, therefore the left hand side
is the $\ep\searrow 0$ limit of
\begin{equation}
-\frac{1}{2d}\sum_{x\in\BZ^d}{G_x^\ep}=\frac{1}{2d}\sum_{x\in\BZ^D}{\ind_{\Om_0}\ind_{\{x\in
N_0(\om)\}}(\psi_e-\psi_\ep\circ\gth_x)}=(1-\Lambda)\psi_\ep.
\end{equation}
Using the definition of $\psi_\ep$ \eqref{defining_psi_ep}, we get
that $(1-\Lambda)\psi_\ep=V-\ep\psi_\ep$. From here, using
\eqref{psi_squared_limit}, we get that the $\ep\searrow 0$ limit is
indeed V in $L^2$.

Finally, we need to show the square integrability
\eqref{square_integrability}. We note that, by the construction of
the corrector,
\begin{equation}
\left[\chi(x+y,\cdot)-\chi(x,\cdot)\right]\ind_{\{x\in\CP(\om)\}}\ind_{\{y\in
N_0(\om)\}}\circ\gth_x=G_{x,x+y}.
\end{equation}
But $G_{x,x+y}$ is the $L^2$ limit of $L^2$-functions
$G_y^{(\ep)}\circ\gth_x$ whose $L^2$ norm is bounded by that of
$G_y^{\ep}$. Hence \eqref{square_integrability} follows with
$C=\max_{\{x:x\in N_0{\om}\}}{\|G_{0,x}\|_2}$.
\end{proof}



\section{Sublinearity along coordinate directions}

We are now ready to start treating the main difficulty of the high
dimensional Central limit theorem proof: the sublinearity of the
corrector. In this section, we treat the sublinearity along the
coordinate directions in $\BZ^d$. Fix $e\in\CE$. We define a sequence $n^e_k(\om)$
inductively by $n^e_1(\om)=f_e(\om)$ and
$n^e_{k+1}=n^e_k(\si_e(\om))$ where $\si_e$ is the induced
translation defined by $\si_e=\gth^{f_e(\om)}_e$. The numbers
$n^e_k$ are well-defined and finite almost surely. Let $\chi$ be the
corrector defined in Theorem \ref{corrector_thm}. The main goal
of this section is to prove the following theorem:

\begin{thm}
For $P$-almost all $\om\in\Om_0$
\begin{equation}
\lim_{k\rightarrow\infty}{\frac{\chi(n^e_k(\om)e,\om)}{k}}=0.
\end{equation}
\label{coordinate_sublinearity}
\end{thm}

The proof of this theorem is based on the following properties of
$\chi(n^e_k(\om)e,\om)$:

\begin{prop}
~\\
\begin{enumerate}

\item $\BE_P\big[|\chi(n^e_1(\om)e,\cdot)|\big]<\infty.$

\item $\BE_P\big[\chi(n^e_1(\om)e,\cdot)\big]=0.$

\end{enumerate}
\label{chi_moment_prop}
\end{prop}

\begin{proof}
Using the definition of the corrector \eqref{chi_definition}, it
follows that
\begin{equation}
\chi(n^e_1(\om)e,\om)=G_{0,n^e_1(\om)e}(\om).
\end{equation}
By \eqref{lim_of_G_exist}, and since $G_{0,n^e_1(\om)e}(\om)$ is
the $\ep\searrow 0$ limit of $G^{(\ep)}_{n^e_1(\om)e}$ in $L^2$, it
follows that $G_{0,n^e_1(\om)e}(\om)\in L^2$. Since $P$ is a
probability measure, it is in particular a finite measure, and
therefore for every $1\leq r<2$ it is also true that
$G_{0,n^e_1(\om)e}(\om)\in L^r$. Taking $r=1$ we find:
\begin{equation}
\BE_P\big[|\chi(n^e_1(\om)e,\cdot)|\big]=\BE_P\big[|G_{0,n^e_1(\om)e}(\om)|\big]<\infty.
\end{equation}

In order to prove part (2), we again use the fact that
$G_{0,n^e_1(\om)e}(\om)$ is the $\ep\searrow 0$ limit in $L^2$ of
$G^{(\ep)}_{n^e_1(\om)e}$, and therefore it's enough to show that
for every $\ep>0$
\begin{equation}
\BE_P\big[G^{(\ep)}_{n^e_1(\om)e}\big]=0.
\end{equation}
and indeed
\begin{equation}
\begin{aligned}
\BE_P\big[G^{(\ep)}_{n^e_1(\om)e}\big]&=\BE_P\big[\ind_{\Om_0}\ind_{\{n^e_1(\om)e\in
N_0(\om)\}}(\psi_\ep\circ\gth^{n^e_1(\om)}_e-\psi_\ep)\big]\nonumber\\
&=\BE_P\big[\ind_{\Om_0}\ind_{\{n^e_1(\om)e\in
N_0(\om)\}}\psi_\ep\circ\gth^{n^e_1(\om)}_e\big]-\BE_P\big[\ind_{\Om_0}\ind_{\{n^e_1(\om)e\in
N_0(\om)\}}\psi_\ep\big]\nonumber\\
&=\BE_P\big[(\ind_{\Om_0}\ind_{\{n^e_1(\om)e\in
N_0(\om)\}}\psi_\ep)\circ\si_e\big]-\BE_P\big[\ind_{\Om_0}\ind_{\{n^e_1(\om)e\in
N_0(\om)\}}\psi_\ep\big],
\end{aligned}
\end{equation}
which equals zero by Theorem \ref{induced_shift_ergodicy} and the
fact that $\psi_\ep$ is absolutely integrable since it is in
$L^2$.
\end{proof}

\begin{proof}[Proof of Theorem \ref{coordinate_sublinearity}]
Let $g:\Om\rightarrow\BR^d$ be defined by
$g(\om)=\chi(n^e_1(\om)e,\om)$, and let $\si_e$ be the induced shift
in direction e. Then
\begin{equation}
\chi(n^e_k(\om)e,\om)=\sum_{i=0}^{k-1}{g\circ\si_e^i(\om)}.
\end{equation}
Using Proposition \ref{chi_moment_prop} we have that $g\in L^1$ and
$\BE_P[g]=0$. Since Theorem \ref{induced_shift_ergodicy} ensures
$\si_e$ is $P$-preserving and ergodic, the claim follows from
Birkhoff's Ergodic Theorem.
\end{proof}


\section{Sublinearity everywhere}

\begin{defn}
Given $K > 0$ and $\ep > 0$, we say that a site $x\in\BZ^d$ is $K,
\ep$-good in configuration $\om\in\Om$ if $x\in\CP(\om)$ and
\begin{equation}
|\chi(y,\om)-\chi(x,\om)|<K+\ep |x-y|,
\end{equation}
holds for every $y\in\CP(\om)$ of the form $y=le$, where $l\in\BZ$
and e is a unit coordinate vector. We will use
$\mathcal{G}_{K,\ep}=\mathcal{G}_{K,\ep}(\om)$ to denote the set of
$K,\ep$-good sites in configuration $\om$.
\end{defn}

\begin{thm}
For every $\ep>0$ and $P$-almost every $\om\in\Om_0$
\begin{equation}
\limsup_{n\rightarrow\infty}{\frac{1}{(2n+1)^d}\sum_{x\in\CP(\om),~|x|\leq
n}\ind_{\{|\chi(x,\om)|\geq \ep n\}}}\leq \ep.
\label{sublinearity_formula}
\end{equation}
\label{sub_linearity_everywhere}
\end{thm}



Before stating the proof, we give a short introduction of the basic
idea. This proof is a light modification of the proof from \cite{BB06}.

Fix the dimension $d$, and for each $\nu = 1,2,\ldots, d$ let
$\Lambda_n^\nu$ be the $\nu$-dimensional box
\begin{equation}
\Lambda_n^\nu=\{k_1 e_1+\ldots+k_\nu e_\nu:k_i\in\BZ,|k_i|\leq
n,\forall i=1,2,\ldots,\nu\}.
\end{equation}
We will run an induction over $\nu$-dimensional sections of the
$d$-dimensional box $\{x\in\BZ^d:|x|\leq n\}$. The induction
eventually gives Theorem \ref{sub_linearity_everywhere} for $\nu=d$
thus proving it. Since it is not advantageous to assume that
$0\in\CP(\om)$, we will carry out the proof for differences of the
form $\chi(x,\om)-\chi(y,\om)$ with $x,y\in\CP(\om)$. For each
$\om\in\Om$, we thus consider the (upper) density
\begin{equation}
\CQ_\nu(\om)=\lim_{\ep\downarrow
0}{\lim_{n\rightarrow\infty}{\inf_{y\in
\CP(\om)\cap\Lambda_n^1}{\frac{1}{|\Lambda_n^1|}{\sum_{x\in\CP(\om)\cap\Lambda_n^\nu}{\ind_{\{|\chi(x,\om)-\chi(y,\om)|\geq
\ep n\}}}}}}}. \label{Q_nu_definition}
\end{equation}
Note that the infimum is taken only over sites in the
one-dimensional box $\Lambda_n^1$. Our goal is to show by induction
that $\CQ_\nu=0$ almost surely for all $\nu=1,\ldots,d$. The
induction step is given by the following lemma:

\begin{lem}
Let $1\leq \nu < d$. If $\CQ_\nu=0$ $P$-almost surely, then also
$\CQ_{\nu+1}=0$ $P$-almost surely. \label{induction_lemma}
\end{lem}

Before we start the formal proof, we give the main idea: Suppose
that $\CQ_\nu=0$ for some $\nu<d$ $P$-almost surely. Pick $\ep>0$.
Then for $P$-almost every $\om$ and all sufficiently large $n$,
there exists a set of sites $\Delta\subset\Lambda_n^\nu\cap\CP(\om)$
such that
\begin{equation}
|(\Lambda_n^\nu\cap\CP(\om))\backslash\Delta|\leq
\ep|\Lambda_n^\nu|, \label{Lambda_1+property}
\end{equation}
and
\begin{equation}
|\chi(x,\om)-\chi(y,\om)|\leq\ep n~~~~\forall ~x,y\in\Delta.
\label{induction_hypothesis}
\end{equation}
Moreover, for $n$ sufficiently large, $\Delta$ could be picked so
that $\Delta\cap\Lambda_n^1\neq\emptyset$ and, assuming $K\gg 1$ the
non-$K,\ep$-good sites could be pitched out with little loss of
density to achieve even
\begin{equation}
\Delta\subset \mathcal{G}_{K,\ep}. \label{Lambda_3_property}
\end{equation}

(All these claims are direct consequences of the Pointwise ergodic
Theorem and the fact that $P(0\in\mathcal{G}_{K,\ep})$ converges to
$P(0\in \Om)$ as $k\rightarrow\infty$.)\\

As a result of this construction we have
\begin{equation}
|\chi(z,\om)-\chi(x,\om)|\leq K+\ep n, \label{bound_sublinearity}
\end{equation}

for any $x\in\Delta$ and any $z\in \Lambda_n^{\nu+1}\cap\CP(\om)$ of
the form $x+je_{\nu+1}$. Thus, if
$r,s\in\CP(\om)\cap\Lambda_n^{\nu+1}$ are of the form,
$r=x+je_{\nu+1}$ and $s=y+ke_{\nu+1}$, then
\eqref{bound_sublinearity} implies
\begin{equation}
\begin{aligned}
|\chi(r,\om)-\chi(s,\om)|&\leq
|\chi(r,\om)-\chi(x,\om)|+|\chi(x,\om)-\chi(y,\om)|+|\chi(y,\om)-\chi(s,\om)|\\
&\leq 2K+2\ep n + |\chi(x,\om)-\chi(y,\om)|.
\label{rs-inequality}
\end{aligned}
\end{equation}

Invoking the induction hypothesis \eqref{induction_hypothesis}, the
right hand side is less than $2K+3\ep n$, implying a bound of the
type \eqref{induction_hypothesis} but one dimension higher.
Unfortunately, the above is not sufficient to prove
\eqref{induction_hypothesis} for all but a vanishing fraction of
sites in $\Lambda_n^{\nu+1}$. The reason is that the $r's$ and $s's$
for which \eqref{rs-inequality} holds, need to be of the form
$x+je_{\nu+1}$ for some $x\in \Delta\cap\CP(\om)$. But $\CP(\om)$
will occupy only about a $P(0\in\CP(\om))$ fraction of all sites in
$\Lambda_n^\nu$, and so this argument does not permit to control
more than a fraction of about $P(0\in \CP(\om))$ of
$\Lambda_n^{\nu+1}\cap\CP(\om)$.

To fix this problem, we will have to work with a "stack" of
translates of $\Lambda_n^\nu$ simultaneously. Explicitly, consider
the collection of $\nu$-boxes
\begin{equation}
\Lambda_{n,j}^\nu=\gth^j_{e_{\nu+1}}(\Lambda_n^\nu)~~~~j=1,2\ldots,L.
\end{equation}
Here $L$ is a deterministic number chosen so that, for a given
$\delta>0$, the set
\begin{equation}
\Delta_0=\{x\in\Lambda_n^\nu:\exists
j\in\{0,1,\ldots,L-1\},~x+je_{\nu+1}\in\Lambda_{n,j}^\nu\cap\CP(\om)\},
\end{equation}
is so large that for sufficiently large $n$
\begin{equation}
|\Delta_0|\geq (1-\delta)|\Lambda_n^\nu|.
\end{equation}
These choices ensure that $(1-\delta)$-fraction of $\Lambda_n^\nu$ is now "covered" which, by
repeating the above argument, gives us control over $\chi(r,\om)$
for nearly the same fraction of all sites $r\in
\Lambda^{\nu+1}\cap\CP(\om)$.

\begin{proof}[Proof of Lemma \ref{induction_lemma}]
Let $\nu<d$ and suppose that $\mathcal{Q}_\nu=0$ $P$-almost surely.
Fix $\delta>0$ with $0<\delta<\frac{1}{2}P(0\in\CP(\om))^2$ and let
$L$ be as defined above. Choose $\ep>0$ so that
\begin{equation}
L\ep+\delta<\frac{1}{2}P(0\in\CP(\om))^2.
\label{delta_epsilon_assumption}
\end{equation}
For a fixed but large $K$, $P$-almost every $\om$ and $n$ exceeding
an $\om$-dependent quantity, for each $j=1,2,\ldots,L$, we can find
$\Delta_j\subset \Lambda_{n,j}^\nu\cap\CP(\om)$ satisfying the
properties (\ref{Lambda_1+property}-\ref{Lambda_3_property}) - with
$\Lambda_n^\nu$ replaced by $\Lambda_{n,j}^\nu$. Given
$\Delta_1,\ldots,\Delta_L$, let $\Lambda$ be the set of sites in
$\Lambda_n^{\nu+1}\cap\CP(\om)$ whose projection onto the linear
subspace $\mathbb{H}=\{k_1 e_1+\ldots+k_\nu e_\nu:k_i\in\BZ\}$
belongs to the corresponding projection of
$\Delta_1\cup\ldots\cup\Delta_L$. Note that the $\Delta_j$ could be
chosen so that $\Lambda\cap\Lambda_n^1\neq\emptyset$. By their
construction, the projections of the $\Delta_j's$, $j=1,\ldots,L$
onto $\mathbb{H}$ "fail to cover" at most $L\ep|\Lambda_n^\nu|$
sites in $\Delta_0$, and so at most $(\delta+L\ep)|\Lambda_n^\nu|$
sites in $|\Lambda_n^\nu|$ are not of the form $x+ie_{\nu+1}$ for
some $x\in\bigcup_j{\Delta_j}$. It follows that
\begin{equation}
|(\Lambda_n^{\nu+1}\cap\CP(\om))\backslash\Lambda|\leq(\delta+L\ep)|\Lambda_n^{\nu+1}|,
\label{set_bound}
\end{equation}
i.e, $\Lambda$ contains all except at most $(\delta+L\ep)$-fraction
of all sites in $\Lambda_n^{\nu+1}$ that we care about. Next we note
that if $K$ is sufficiently  large, then for every $1\leq i<j\leq
L$, the set $\mathbb{H}$ contains
$\frac{1}{2}P(0\in\CP(\om))$-fraction of sites such that
\begin{equation}
z_i\overset{def}{=}x+ie_\nu\in\mathcal{G}_{K,\ep},~~~~~~z_j=x_je_\nu\in\mathcal{G}_{K,\ep}.
\end{equation}
Since we assumed \eqref{delta_epsilon_assumption}, once $n\gg 1$,
for each pair $(i,j)$ with $1\leq i<j\leq L$ such $z_i$ and $z_j$
can be found so that $z_i\in\Delta_i$ and $z_j\in\Delta_j$. But the
$\Delta_j's$ were picked to make \eqref{induction_hypothesis} true
and so using these pairs of sites we now show that
\begin{equation}
\begin{aligned}
|\chi(y,\om)-\chi(x,\om)|&\leq
|\chi(y,\om)-\chi(z_j,\om)|+|\chi(z_j,\om)-\chi(z_i,\om)|+|\chi(z_i,\om)-\chi(x,\om)|\\
&\leq \ep n + K + \ep L + \ep n = K+\ep L+2\ep n,
\label{induction_inequality}
\end{aligned}
\end{equation}
for every $x,y\in\Delta_1\cup\ldots\cup\Delta_L$.
From \eqref{induction_hypothesis} and \eqref{induction_inequality},
we now conclude that for all $r,s\in\Lambda$,
\begin{equation}
|\chi(r,\om)-\chi(s,\om)|\leq 3K+\ep L + 4\ep n< 5\ep n,
\label{chi_bound}
\end{equation}
assuming that $n$ is so large that $\ep n>3K+\ep L$. If
$\CQ_{\nu,\ep}$ denotes the right-hand side of
\eqref{Q_nu_definition} before taking $\ep\searrow 0$, the bounds
\eqref{set_bound} and \eqref{chi_bound} and the fact that
$\Lambda\cap \Lambda_n^1\neq\emptyset$ yield
\begin{equation}
\CQ_{\nu+1,5\ep}(\om)\leq \delta+L\ep,
\end{equation}
for $P$-almost every $\om$, But the left-hand side of this
inequality increases as $\ep\searrow 0$ while the right hand side
decreases. Thus, taking $\ep\searrow 0$ and $\delta\searrow 0$
proves that $\CQ_{\nu+1}=0$ holds $P$-almost surely.
\end{proof}

\begin{proof}[Proof of Theorem \ref{sub_linearity_everywhere}]
The proof is an easy consequence of Lemma \ref{induction_lemma}.
First, by Theorem \ref{coordinate_sublinearity} we know that
$\CQ(\om)=0$ for $P$-almost every $\om$. Invoking appropriate
shifts, the same conclusion applies $Q$ almost surely. Using
induction on dimension, Lemma \ref{induction_lemma} then tells us
that $\CQ_d(\om)=0$ for $P$ almost every $\om$. Let $\om\in\Om_0$.
By Theorem \ref{coordinate_sublinearity}, for each $\ep>0$ there is
$n_0=n_0(\om)$ with $P(n_0<\infty)=1$ such that for all $n\geq
n_0(\om)$, we have $|\chi(x,\om)|\leq \ep n$ for all $x\in
\Lambda_n^1\cap \CP(\om)$. Using this to estimate away the infimum
in \eqref{Q_nu_definition}, the fact that $\CQ_d=0$ now immediately
implies \eqref{sublinearity_formula} for all $\ep>0$.
\end{proof}



\section{High dimensional Central Limit Theorem}

The theorem we wish to prove in this section is the following:

\begin{thm}
Fix $d\geq 2$. Assume the additional assumption, assumption
\ref{assumption3}, then for $P$ almost every $\om\in\Om_0$
\begin{equation}
\lim_{n\rightarrow\infty}{\frac{X_n}{\sqrt{n}}}\overset{D}{=}N(0,D),
\end{equation}
where $N(0,D)$ is a $d$-dimensional multivariate normal distribution with
covariance matrix $D$ that depends only on $d$ and the distribution of $P$.
\end{thm}

We start with the following lemma:

\begin{lem}
Fix $\om\in\Om_0$ and let $x\mapsto \chi(x,\om)$ be the corrector as
defined in Theorem \ref{corrector_thm}. given a path of a random
walk $\{X_n\}_{n=0}^{\infty}$ on $\CP(\om)$ with transition
probabilities \eqref{transition_probability} let
\begin{equation}
M_n^{(\om)}=X_n+\chi(X_n,\om),~~~~\forall n\geq 0.
\label{deformed_random_walk}
\end{equation}
Then $\{M_n^{(\om)}\}_{n=0}^{\infty}$ is an $L^2$-martingale for the
filtration $\{\si(X_0,X_1,\ldots,X_n)\}_{n=0}^{\infty}$. Moreover,
conditional on $X_{k_0}=x$, the increments
$\{M_{k+k_0}^{(\om)}-M_{k_0}^{(\om)}\}_{k=0}^{\infty}$ have the same
law as $\{M_k^{(\gth_x\om)}\}_{k=0}^{\infty}$.
\end{lem}

\begin{proof}
Since $X_n$ is bounded, $\chi(X_n,\om)$ is bounded and so
$M_n^{(\om)}$ is square integrable with respect to $P_\om$. Since
$x\mapsto x+\chi(x,\om)$ is harmonic with respect to the transition
probabilities of the random walk $(X_n)$ with law $P_\om$ we have
\begin{equation}
E_\om[M_{n+1}^{(\om)}|\si(X_n)]=M_n^{(\om)}~~~\forall n\geq 0, P_\om a.s.
\end{equation}
Since $M_n^{(\om)}$ is $\si(X_n)$-measurable,
$(M_n^{(\om)})$ is a martingale. The stated relation between the
laws of $(M_{k+k_0}^{(\om)}-M_{k_0}^{(\om)})_{k\geq 0}$ and
$(M_k^{(\gth_x \om)})_{k\geq 0}$ is implied by the shift invariance
proved in Theorem \ref{shift_invariance} and the fact that
$(M_n^{(\om)})$ is a simple
random walk on the deformed graph.
\end{proof}

\begin{thm}[The Modified random walk CLT]
Fix $d\geq 2$, and assume in addition, assumption \ref{assumption3}.
For $\om\in\Om_0$ let $\{X_n\}_{n=0}^\infty$ be random walk with
transition probabilities \eqref{transition_probability} and let
$\{M_n^{(\om)}\}_{n=0}^{\infty}$ be as defined in
\eqref{deformed_random_walk}. Then for $P$ almost every $\om\in\Om_0$
we have
\begin{equation}
\lim_{n\rightarrow\infty}\frac{M_n^{(\om)}}{\sqrt{n}}\overset{D}{=}
N(0,D),
\end{equation}
where $N(0,D)$ is a $d$-dimensional multivariate normal distribution
with covariance matrix $D$ which depends only on $d$ and the
distribution $P$, and is given by
$D_{i,j}=\BE\left[cov(M_1^{(\om)}\cdot e_i,M_1^{(\om)}\cdot
e_j)\right]$ \label{Modified_CLT}
\end{thm}

\begin{proof}
Let
\begin{equation}
V_n^{(\om)}(\ep)=\frac{1}{n}\sum_{k=0}^{n-1}{E_\om\left[D_k^{(\om)}
\ind_{\{\min_{i,j}|(D_k^{(\om)})_{i,j}|\geq \ep \sqrt{n}
\}}\Big|X_0,X_1,\ldots,X_k\right]},
\end{equation}
where $D_k^{(\om)}$ is the covariance matrix for
$M_{k+1}^{(\om)}-M_k^{(\om)}$. By the Lindeberg-Feller Central Limit
Theorem (see for example \cite{Du96}),  it is enough to show that\\
\begin{enumerate}
\item $\lim_{n\rightarrow\infty}{V_n^{(\om)}(0)}=D$ in
$P_\om$-probability. \\

\item $\lim_{n\rightarrow\infty}{V_n^{(\om)}(\ep)}=0$ in $P_\om$-probability for all $\ep>0$.\\
\end{enumerate}
Both conditions are implied from Theorem \ref{Thm_mutual_ergodic}.
Indeed
\begin{equation}
V_n^{(\om)}(0)=\frac{1}{n}\sum_{k=0}^{n-1}{h_0\circ\gth_{X_k}(\om)},
\nonumber
\end{equation}
where
\begin{equation}
h_K(\om)=E_\om\left[D_1^{(\om)}
\ind_{\{\min_{i,j}|(D_1^{(\om)})_{i,j}|\geq K \}}\right].
\end{equation}
Therefore by Theorem \ref{Thm_mutual_ergodic} we have for $P$ almost
every $\om\in\Om_0$
\begin{equation}
\lim_{n\rightarrow\infty}{V_n^{(\om)}(0)}=\BE\left[h_0(\om)\right]=D.
\end{equation}
On the other hand, for every $K\in\BR$ and every $\ep>0$ we have
$\ep\sqrt{n}>K$ for sufficiently large $n$, and therefore
$f_{\ep\sqrt{N}}\leq f_K$. So $P$-almost surely
\begin{equation}
\limsup_{n\rightarrow\infty}{V_n^{(\om)}(\ep)}\leq
\BE\left[D_1^{(\om)} \ind_{\{\min_{i,j}|(D_1^{(\om)})_{i,j}|\geq K
\}}\right]\underset{_{K\rightarrow\infty}}{\longrightarrow}~ 0.
\end{equation}
Where in order to apply the Dominated Convergence, we used the fact
that
$M_1^{(\om)}\in L^2$.
\end{proof}

We are now ready to prove the high dimensional Central Limit Theorem

\begin{proof}[Proof of Theorem \ref{CLT2}]
Due to Theorem \ref{Modified_CLT} it is enough to prove that for
$P$-almost every $\om\in\Om_0$
\begin{equation}
\lim_{n\rightarrow\infty}{\frac{\chi(X_n,\om)}{\sqrt{n}}}=0~~~P_\om~a.s.
\end{equation}
This will follow if we will show that there exists a constant $K>0$ such
that for every $\ep >0$ and for $P$-almost every $\om\in\Om_0$
\begin{equation}
\lim_{n\rightarrow\infty}{P_\om\{|\chi(X_n,\om)|>\ep\sqrt{n}\}}<K\ep.
\end{equation}
By Theorem \ref{asymptotic_X_n} and the Markov inequality, there exists a random $c=c(\om>0$, $P$ almost surely finite, such that
that for $P$-almost every $\om\in\Om_0$
\begin{equation}
P_\om\left[\|X_n\|>\frac{1}{\ep}\sqrt{n}\right]\leq
\ep\frac{\BE_\om(\|X_n\|)}{\sqrt{n}}\leq c\ep .\label{Markov_X_n}
\end{equation}
We therefore get
\begin{equation}
\begin{aligned}
P_\om\left(|\chi(X_n,\om)|>\ep\sqrt{n}\right) \nonumber
\leq P_\om\left(\|X_n\|>\frac{\sqrt{n}}{\ep}\right)+
P_\om\left(\chi(X_n,\om)>\ep\sqrt{n},\|X_n\|\leq\frac{\sqrt{n}}{\ep}\right).\nonumber
\end{aligned}
\end{equation}
By \eqref{Markov_X_n} we find that this is less or equal than
\begin{equation}
c\ep + \sum_{x\in
\CP(\om)}P_\om^n(0,x)\ind_{\left\{|\chi(x,\om)|>\ep\sqrt{n},~x\in\left[-\frac{\sqrt{n}}{\ep},\frac{\sqrt{n}}{\ep}\right]\right\}}.
\nonumber
\end{equation}
Using now Theorem \ref{bound_of_transitions} for sufficiently $n$ if
follows that
\begin{equation}
P_\om\left(|\chi(X_n,\om)|>\ep\sqrt{n}\right)\leq c\ep
+\frac{K_1}{(n-K_2)^\frac{d}{2}}\sum_{x\in \CP(\om)\cap
\left[-\frac{\sqrt{n}}{\ep},\frac{\sqrt{n}}{\ep}\right]}{\ind_{\left\{\chi(x,\om)>\ep\sqrt{n}\right\}}}.
\nonumber
\end{equation}
Therefore by Theorem \ref{sub_linearity_everywhere} we get that
there exist constants $c_0,K$ such that
\begin{equation}
\lim_{n\rightarrow\infty}{P_\om(|\chi(X_n,\om)|>\ep\sqrt{n})}\leq c\ep
+c_0\ep^2\leq K\ep \nonumber
\end{equation}
As required.

\end{proof}


\section{Some Conjectures And Questions}

While we have full classification of transience recurrence of random
walks on discrete point processes in dimensions $d=1$ and $d\geq 3$,
we only have a partial classification in dimension 2. We therefore
give the following two conjectures:

\begin{conj}
There are transient two dimensional random walks on discrete point processes.
\end{conj}

\begin{conj}
The condition given in Theorem \ref{tran_recu2}, for recurrence of
2-dimensional random walk on discrete point process, i.e, the
existence of a constant $C>0$ such that
\begin{equation}
\sum_{k=N}^{\infty}{\frac{k\cdot
P(f_{e_i}=k)}{\BE(f_{e_i})}}\leq\frac{C}{N}~~~i\in\{1,2\}~~N\in\BN
\end{equation}
is not necessary.
\end{conj}

In Theorem \ref{CLT2} we gave conditions for the random walk on
discrete point processes to satisfy a Central Limit Theorem.
However, we didn't give any example for a random walk without
a Central Limit Theorem. We therefore give the following conjecture:

\begin{conj}
There are random walks on discrete point processes in high
dimensions
that don't satisfy a Central Limit Theorem.
\end{conj}

In the proof of Theorem \ref{CLT2}, we used the additional
assumption that there exists $\ep_0>0$ such that for every
coordinate direction $e$ $E_P[f_e^{2+\ep_0}]<\infty$. The assumption
that the second moments are finite, is fundamental in our proof in
order to build the corrector, and seems to be necessary for the CLT
to hold. On the other hand, existence of such $\ep_0>0$ though
needed in our proof, was used only in order to bound
\eqref{distance_prop_3}. We therefore give the following condition:

\begin{conj}
Theorem \ref{CLT2} is true even with the weak assumption that only the second moments are finite.
\end{conj}

Even if the theorem is true with the weak assumption that only the
second moment of the distances between points is finite, we can still ask the following question:

\begin{conj}
Is the condition given in Theorem \ref{CLT2} also necessary, or can
one find examples for random walks on discrete point processes that
satisfy a Central Limit Theorem but don't have all of their second
moments finite? We conjecture that such examples exist, but didn't verified it.
\end{conj}

We also have the following conjecture about the Central Limit
Theorem:

\begin{conj}
Under assumptions \ref{Assumptions} and \ref{assumption3}, The
Central Limit Theorem, \ref{CLT2}, can be strengthened as follows:
Random walk on discrete point process under appropriate scaling
converges to Brownian motion.
\end{conj}

Our model describes non nearest neighbors random walk on random
subset of $\BZ^d$ with uniform transition probabilities. We suggest
the following generalization of the model:

\begin{ques}
Fix $\alpha\in\BR$. We look on the same model for the environments
with transition probabilities as follows: for $\om\in\Om_0$
\begin{equation}
P_\om(X_{n+1}=u|X_n=v)=\left\{
\begin{array}{cc}
0&~~~u\notin N_v(\om) \\
\frac{1}{Z(v)}\|u-v\|^\alpha&~~~u\in N_v(\om)
\end{array}
\right. ,
\end{equation}
where $Z(v)$ is normalization constant (The case $\alpha=0$ is the uniform distribution case). What can be proved about the extended model?
\end{ques}


\section*{Appendix}

In this Appendix we prove there exists a constant $c=c(d)>0$ such that for every $0<a\leq 2$
\begin{equation}
\sum_{n=0}^{\infty}e^{-a\cdot 2^n} \leq ca^{-d}
\end{equation}

\begin{proof}
First, we can restrict ourselves to $0<a<\ep$ for any fixed $\ep>0$. This follows from the fact that both expressions are monotonic in $a$. Next we note that:
\begin{equation}
\begin{aligned}
\sum_{n=0}^{\infty}{e^{-a\cdot 2^n}\cdot 2^{nd}}& \leq 1+\sum_{n=1}^\infty \frac{1}{1-2^{-d}}\sum_{k=2^{n-1}d}^{2^n d}e^{-a2^n}
\leq \frac{1}{1-2^{-d}} \sum_{k=0}^\infty e^{-ak^{1/d}}\\
&=\frac{1}{1-2^{-d}} \sum_{j=0}^\infty e^{-aj}\#\{j<k^{1/d}\leq j+1\}
=\frac{1}{1-2^{-d}} \sum_{j=0}^\infty e^{-aj}[(j+1)^d-j^d].\nonumber
\end{aligned}
\end{equation}
Since there exists a constant $c=c(d)>0$ such that for every $j\geq 0$ we have $(j+1)^d-j^d\leq cj^{d-1}$ the last term is less than or equal to
\[\label{appendix_a1}
\frac{c}{1-2^{-d}}\sum_{j=0}^\infty e^{-aj}j^{d-1}.
\]
For $j\geq 0$ denote $\alpha_j=e^{-aj}j^{d-1}$ and define
$j_0=\min\left\{j\geq 0 ~:~ \forall i\geq j~~\frac{a_{i+1}}{a_i}<e^{-a/2}\right\}.$
From the definition of $j_0$ it follows that \eqref{appendix_a1} is less than
\begin{equation}\label{appendix_a2}
\begin{aligned}
\frac{c}{1-2^{-d}}\left[\sum_{j=0}^{j_0-1} a_j+\sum_{j=j_0}^\infty a_j \right]&\leq \frac{c}{1-2^{-d}}\left[\sum_{j=0}^{j_0-1} a_j+\sum_{j=j_0}^\infty a_{j_0}e^{-a(j-j_0)/2} \right].\\
&\leq \frac{c}{1-2^{-d}}\left[j_0^{d-1}+\frac{\alpha_{j_0}}{1-e^{-a/2}}\right].
\end{aligned}
\end{equation}
From the definition of $j_0$ one can see that $j_0=\left\lceil\frac{1}{e^{\frac{a}{2d}}-1}\right\rceil\leq \left\lceil\frac{2d}{a}\right\rceil$, and therefore \eqref{appendix_a1} equals to
\begin{equation}
\frac{c}{1-2^{-d}}\left(1+\frac{e^{-a \left\lceil\frac{2d}{a}\right\rceil}}{\frac{a}{2}e^{-a/2}}\right)\left\lceil\frac{2d}{a}\right\rceil^{d-1}
\end{equation}
which for an appropriate constant $c=c(d)>0$ is less than $ca^{-d}$, as required.
\end{proof}


\bibliography{pointproc}
\bibliographystyle{alpha}


\end{document}